\pdfmapfile{=<name>.map}
\documentclass[12pt,oneside,reqno,fleqn]{amsart}
\usepackage{amssymb}
\usepackage{bbm}
\usepackage{cases}
\usepackage{amsmath,mathtools}
\usepackage{graphicx}
\usepackage{mathrsfs}
\usepackage{stmaryrd}
\usepackage[polish,english]{babel}  
\usepackage{color}
\usepackage{soul}
\usepackage[dvipsnames]{xcolor}
\usepackage{amsfonts}
\usepackage{amsmath,amssymb,amsthm}
\usepackage{libertine}
\usepackage[T1]{fontenc}
\usepackage[libertine,varbb]{newtxmath}
\usepackage[colorlinks,linkcolor=Red,citecolor=Red,urlcolor=Red]{hyperref}
\usepackage[shortlabels] {enumitem}
\usepackage[nameinlink,capitalize]{cleveref}
\usepackage{todonotes}
\usepackage[spacing=true,kerning=true,tracking=true,letterspace=50]{microtype}
\usepackage{orcidlink}
\usepackage{ulem}
\usepackage{verbatim}
\usepackage[T1]{fontenc}  
\usepackage[utf8]{inputenc}  
\hypersetup{pdftitle={Regularization by regular noise-numerical approximation}, pdfauthor={Chengcheng Ling, Ke Song and Haiyi Wang}}
\numberwithin{equation}{section}
\newcommand{\be}{\begin{eqnarray}}
\newcommand{\ee}{\end{eqnarray}}
\newcommand{\ce}{\begin{eqnarray*}}
\newcommand{\de}{\end{eqnarray*}}
\newtheorem{theorem}{Theorem}[section]
\newtheorem{lemma}[theorem]{Lemma}
\newtheorem{proposition}[theorem]{Proposition}
\newtheorem{corollary}[theorem]{Corollary}
\theoremstyle{remark}
\newtheorem{assumption}[theorem]{Assumption}

\newtheorem{example}[theorem]{Example}
\newtheorem{remark}[theorem]{Remark}
\newtheorem{definition}[theorem]{Definition}

\crefname{eqn}{Equation}{Equations}
\crefname{assumption}{Assumption}{Assumptions}
\crefname{innercustomthm}{Condition}{Conditions}
\crefrangelabelformat{innercustomthm}{#3#1#4-#5#2#6}

\def\={&\!\!=\!\!&}

\def\eps{\varepsilon}

\def\<{{\langle}}
\def\>{{\rangle}}
\def\({{\Big(}}
\def\){{\Big)}}

\def\bx{{\mathbf{x}}}

\def\={&\!\!=\!\!&}
\def\bt{\begin{theorem}}
\def\et{\end{theorem}}
\def\bl{\begin{lemma}}
\def\el{\end{lemma}}
\def\br{\begin{remark}}
\def\er{\end{remark}}
\def\bd{\begin{definition}}
\def\ed{\end{definition}}
\def\bp{\begin{proposition}}
\def\ep{\end{proposition}}
\def\bc{\begin{corollary}}
\def\ec{\end{corollary}}
\def\bx{\begin{example}}
\def\ex{\end{example}}
\def\cA{{\mathcal A}}

\def\cC{{\mathcal C}}

\def\cE{{\mathcal E}}
\def\cF{{\mathcal F}}

\def\cH{{\mathcal H}}

\def\cL{{\mathcal L}}

\def\cP{{\mathcal P}}

\def\mE{{\mathbb E}}
\def\E{\mE}

\def\mN{{\mathbb N}}

\def\mP{{\mathbb P}}

\def\mR{{\mathbb R}}

\def\mZ{{\mathbb Z}}

\def\geq{\geqslant}
\def\leq{\leqslant}

\newcommand{\R}{{\mathbb R}}

\newcommand{\norm}[1]{{\left\vert\kern-0.25ex\left\vert\kern-0.25ex\left\vert #1 
    \right\vert\kern-0.25ex\right\vert\kern-0.25ex\right\vert}}

\renewcommand{\le}{\leq}
\renewcommand{\ge}{\geq}
\usepackage{url}
\allowdisplaybreaks
\begin{document}
		\title{Regularization by regular noise: a numerical result}
 \author{ Ke Song,  Chengcheng Ling \and  Haiyi Wang 
}
	\date{\today}
    \address{Beijing Institute of Technology, School of Mathematics and Statistics, China.}
\email{ske2022@126.com}
 \address{University of Augsburg, Institut f\"ur Mathematik, 
86159 Augsburg,  Germany.}
\email{chengcheng.ling@uni-a.de}

\address{Peking University, School of Mathematical Sciences, China.}
\email{wanghaiyi@stu.pku.edu.cn}

\begin{abstract}
We study a singular stochastic equation driven by regular noise of fractional Brownian type with Hurst index $H \in (1,\infty)\setminus\mathbb{Z}$ and drift coefficient $b \in \mathcal{C}^\alpha$, where $\alpha > 1 - \tfrac{1}{2H}$. The strong well-posedness of this equation was first established in \cite{GG}, a phenomenon known as {\it regularization by regular noise}. In this note, we provide a numerical analysis of the equation. Specifically, we prove that the Euler--Maruyama approximation $X^n$ converges strongly to the unique solution $X$ at rate $n^{-1}$. Moreover, we show that $n(X - X^n)$ converges in probability to a non-trivial limit as $n \to \infty$, which confirms that the rate $n^{-1}$ is optimal for this scheme. In this sense, this provides a first-order numerical method for equations with non-Lipschitz drift while still achieving the rate $n^{-1}$.

	\bigskip
		

	\end{abstract}
	
	\maketitle

\setcounter{tocdepth}{2}
\section{Introduction}
It is known that when $H\in(0,1)$, a $d$-dimensional fractional Brownian motion (fBM) $B^H$ can be defined via the
Mandelbrot--van Ness representation \cite{MVN}:
\begin{align}\label{def:fBM-Hsmall1}
    B_t^H:=\int_{-\infty}^0(|t-s|^{H-\frac{1}{2}}-|s|^{H-\frac{1}{2}})\mathrm{d} W_s+\int_0^t|t-s|^{H-\frac{1}{2}}\mathrm{d} W_s, \quad t\geq 0,
\end{align}
where $W$ is a two-sided $d$-dimensional standard Brownian motion on some probability
space $(\Omega,\cF,\mP)$. As discussed in \cite{GG}, the fractional integral in \eqref{def:fBM-Hsmall1} can in fact be naturally extended to the regime where the Hurst parameter $H>1$, that is, for $H\in (1,\infty)\backslash\mZ$, 
\begin{align}
   \label{def:fbmHbig1} 
   B_t^H:=\int_{0\leq s_1\leq\ldots\leq s_{\lfloor H\rfloor}\leq t}B^{H-\lfloor H\rfloor}_{s_1}\mathrm{d} {s_1}\ldots \mathrm{d}  s_{\lfloor H\rfloor}.
\end{align}
Due to the possible multiple integrals in \eqref{def:fbmHbig1}, we can see clearly that the paths of $B^H$ for $H>1$ are regular, at least $C^1$. This fact leads to 
one of the appealing results of \cite{GG}, 
which shows the strong well-posedness of the following equation with singular $b$:
\begin{align}\label{eq:SDE}
\mathrm{d} X_t=b(X_t)\mathrm{d} t+\mathrm{d} B_t^H,\quad X_0=x_0\in\mR^d,
\end{align}
where $b\in C^\alpha(\mR^d;\mR^d)$ with $\alpha\in(1-\frac{1}{2H},1)$ and $B^H$ is a $d$-dimensional fBM with $H\in(1,\infty)\backslash\mZ$. 
This result supports the principle of {\it regularization by noise: increased noise roughness leads to enhanced regularization}(\cite{gubicat, harang2020cinfinity, ling, Le, MP, MM, GaGe}). There is also an alternative way, mentioned in \cite{GG}, to interpret \eqref{eq:SDE} as a singular coupled equation perturbed by degenerate noise:
\begin{align}\label{eq:SDEcoup}
   \left\{ \begin{array}{cc}
       \mathrm{d}  X_t  & =\big(b(X_t)+V_t^{\lfloor H\rfloor}\big) \mathrm{d} t,\\
         \mathrm{d} V_t^{\lfloor H\rfloor}&=V_t^{\lfloor H\rfloor-1}\mathrm{d} t,\\
       \qquad\qquad\qquad\qquad \ldots, &\\
        \mathrm{d} V_t^1&=\mathrm{d} B_t^{H-\lfloor H\rfloor},
    \end{array}\right.
\end{align}
which shares a similar framework with \cite{RM, HZZZ, HLL, Ling24}, but with non-Markovian noise when $H\neq k+\frac{1}{2}$ for any $k\in\mN$.

Our motivation here is to provide a numerical result for this equation by
considering its Euler--Maruyama (EM) scheme
\begin{align}
    \label{eq:SDE-EM}
   \mathrm{d} X_t^n=b(X_{k_n(t)}^n)\mathrm{d} t+\mathrm{d} B_t^H,\qquad X_0^n=x_0^n\in\mR^d
\end{align}
with $k_n(t)\coloneqq\frac{\lfloor nt\rfloor}{n}$.

When $H\in (0,1)$, \cite{BDG} has shown the strong convergence rate $n^{-(\frac{1}{2}+\alpha H)\wedge 1+\epsilon}$ of the scheme \eqref{eq:SDE-EM} to \eqref{eq:SDE}, in which the Girsanov Theorem and the Stochastic Sewing Lemma (SSL) \cite{Le} play crucial roles. More precisely, the methodology
in \cite{BDG} can be roughly summarized as follows (taking $x_0=x_0^n$):
\begin{align*}
    \|\sup_{t\in[0,1]}|X_t-X^n_t|\|_{L^p_\omega}\overset{\text{Girsanov}}{\lesssim}\big\|\int_0^1b(B_{s}^H)-b(B_{k_n(s)}^H)\mathrm{d} r\big\|_{L^p_\omega}\overset{\text{SSL}}{\lesssim} n^{-(\frac{1}{2}+\alpha H)\wedge 1+\epsilon}.
\end{align*}

When $H>1$, on the one hand, as indicated in \cite{GG, GaGe}, the Girsanov Theorem becomes less helpful; meanwhile, PDE tools clearly do not apply. Therefore, the challenge in showing the convergence of \eqref{eq:SDE-EM} to \eqref{eq:SDE} lies in two aspects compared with known results: the absence of Girsanov’s theorem and the non-Markovian nature of the noise. Alternatively, although \cite{GG} studies only well-posedness, it already hints at a possible way to avoid using Girsanov’s theorem; concerning numerical approximation, \cite{BDG-Levy} provides similar evidence, but it considers singular SDEs driven by an $\alpha$-stable process, which is Markovian. Nevertheless, we are able to show a convergence rate of $n^{-1}$, which is comparable with \cite{BDG}, since $H>1$ here implies $\alpha>1-\frac{1}{2H}>\frac{1}{2}$.

Given this convergence rate, it is natural to ask how far it is from being optimal. Such questions on the optimality of the EM scheme have been addressed, for instance, in \cite{TY, EGY, DGL-CLT} for Brownian noise, \cite{PSS} for L\'evy processes with jumps, and \cite{AN, NN, JYS, HLN} for fBM with $H<1$. In particular, \cite{AN} confirms that for $H\in(\frac{1}{2},1)$, the order $n^H$ is optimal for the EM scheme \eqref{eq:SDE-EM} when $b\in \cC^2$ (twice differentiable) with possible linear growth. Here, we complete this result with rate $n^{-1}$ for $H>1$ and less regular $b$ (in fact, only $\cC^1$). The idea of the proof is straightforward: we show that for $b\in\cC^\alpha$, the following approximation converges in probability to a possibly non-zero limit:
\begin{align*}
    n(X_t-X^n_t)\overset{n\rightarrow\infty}{\rightarrow} c(t)\neq0,
\end{align*}
which indicates that for \eqref{eq:SDE-EM}, the best convergence rate one can expect is no faster than order $n^{-1}$. Consequently, this also verifies that the rate we obtain is indeed optimal.
Evidently, \eqref{eq:SDE-EM} provides a stochastic numerical method for simulating equations with singular $b$ while still converging with rate $n^{-1}$. 

Lastly, we mention a few related works that share a similar interest and spirit in the study of numerical approximations. For equations of the type \eqref{eq:SDE} with singular drift, convergence results have been established in \cite{NS, DG, JM, BW} for additive Brownian motion, in \cite{DGL, BDG, GLL, LL} for multiplicative Brownian noise, in \cite{BDG-Levy,  BWWZ} for L\'evy processes with jumps, and in \cite{BDG, GHR} for fractional Brownian motion. A slightly different notion of singularity-namely, piecewise Lipschitz coefficients-has also been investigated, with convergence results obtained in works such as \cite{LS, TL, mrl, PSS}. We emphasize that this is only a brief selection from a vast body of literature, and we encourage readers to consult the cited works for further details and insights.

\subsection*{Organization of the paper}
In \cref{sec:Notation-result} we introduce the necessary notations and main results. \cref{sec:auxi-tool} collects all of the crucial tools and properties of fBM. We present the central analysis and the proof of the main convergence result in \cref{sec:proof}. Finally, we show optimality in \cref{sec:optimal}. \cref{app} contains several auxiliary proofs.

\section{Preliminaries and main results}\label{sec:Notation-result}
\subsection{Preliminaries}

On finite dimensional vector spaces we always use the Euclidean norm.

For $k\in\mathbb{N}$, $f:\mathbb{R}^d\mapsto \mathbb{R}$, denote $\partial_k f(x)\coloneqq\frac{\partial f(x)}{\partial x_k}$ for $x\in\mathbb{R}^d$ and   $\nabla f(x)\coloneqq(\partial_if(x))_{1\leq i\leq d}$, the derivative is understood in the weak sense.
For vector-valued $f$ we use the same notation, and
$\nabla^k f$ is defined via $\nabla(\nabla ^{k-1}f)$ iteratively. For a multi-index $k=(k_1,\ldots,k_d)\in\mN^d$, denote $\partial^k f(x)\coloneqq\frac{\partial^{|k|} f(x)}{\partial x_{k_1}\cdots\partial x_{k_d}}$. If $k=(0,\ldots,0)$, we use convention $\partial^kf=f$.
We denote by ${\cC}^\infty_0$ (${\cC}_p^\infty$, resp.) the set of all continuously infinitely differentiable functions that, along with all of their partial derivatives, are compactly supported (of polynomial growth, resp.).

For $\alpha\in(0,1)$,  we set $\mathcal{C}^\alpha(\mR^d)$ to be the space of continuous functions such that
\begin{align*}
  \Vert f\Vert_{\mathcal{C}^\alpha}\coloneqq[f]_{\mathcal{C}^\alpha}+\sup_{x\in\mR^d}|f(x)|\coloneqq\sup_{x,y\in\mR^d,x\neq y}\frac{|f(x)-f(y)|}{|x-y|^\alpha}+\sup_{x\in\mR^d}|f(x)|<\infty.
\end{align*}
    Here, and often below, we write $\mathcal{C}^\alpha$ instead of $\mathcal{C}^\alpha(\mR^d)$ for simplicity.
    For $\alpha\in(0,\infty)$, we define $\mathcal{C}^\alpha(\mR^d)$ the space of all functions $f$ defined on $\mR^d$ having bounded derivatives $\partial^k f$ for multi-indices $k\in\mN^d$ with $|k|\leq \alpha$ so that
        \begin{align*}
            \Vert f\Vert_{\mathcal{C}^\alpha}&\coloneqq\|f\|_{\mathcal{C}^{ \lfloor\alpha\rfloor}}+[f]_{\mathcal{C}^\alpha}\coloneqq\sum_{|k|\leq \lfloor\alpha\rfloor}\sup_{x\in\mR^d}|\partial^kf(x)|+\sum_{k=\lfloor\alpha\rfloor} [\partial^kf]_{\mathcal{C}^{\{\alpha\}}}<\infty,
        \end{align*}
    where $\{\alpha\}:=\alpha-\lfloor\alpha\rfloor$. Note that the $\mathcal{C}^\alpha$-norm always includes the supremum of the function.
 We also denote the space of bounded measurable functions $\mathcal{C}^0(\mR^d)$  with the supremum norm.
    To be noticed that the functions in $\mathcal{C}^0$ do not need to be continuous.

Let $n\in\mN$ and $\alpha,\beta\in(0,1)$ such that $\alpha+\beta>1$. Then for $f\in \cC^\alpha([0,T],\R^n)$, $g\in \cC^\beta([0,T],\R^n)$, the {\it Young integral}
 $h_t=\int_{0}^{t}f_tdg_t$
is well-defined as the limit as $m\rightarrow\infty$ of the Riemann sums
\begin{align*}
  \sum_{i=0}^{m}f_{t_i^m}(g_{t_{i+1}^m\wedge t}-g_{t_{i}^m\wedge t})
\end{align*}
where $(t_i^m)_{0\leq m}$ is any partition sequence of $[0,t]$. And the Young integral satisfies the estimates: there exists the constant $C$ depends only on $\alpha,\beta$ so that for all $0\leq s\leq t\leq T,$
\begin{align*}
  |h_t-h_s-f_s(g_t-g_s)|\leq C|t-s|^{\alpha+\beta}[f]_{\cC^\alpha([s,t])}[g]_{\cC^\beta([s,t])},
\end{align*}
which also yields the following
\begin{align}\label{Young-est}
[h]_{\cC^\beta([s,t])}\leq C\Vert f\Vert_{\cC^\alpha([s,t])}[g]_{\cC^\beta([s,t])}.
\end{align}
        
    In the following we denote the conditional expectation w.r.t. the $\sigma$-algebras of the filtration $(\mathcal{F}_t)_{t\geq0}$ as $\mE^t(\cdot)\coloneqq\mE(\cdot|\mathcal{F}_t), t\geq0$, $\|X\|_{L^p_\omega}:=(\mE|X|^p)^\frac{1}{p}$, $\|X\|_{L^p_\omega|\cF_s}:=(\mE[|X|^p|\cF_s])^\frac{1}{p}$. 
    
    For $p \in [1,\infty]$, $X \in L^p(\Omega,\mR^d)$ and $\mathcal{F}_s$-measurable $\mR^d$ valued random vector $Y$, we have the following inequalities 
        \begin{equation}\label{eq:CJI}
            \|\mE^s X\|_{L^p_\omega} \le \| X\|_{L^p_\omega}
        \end{equation}
    and 
      \begin{equation} \label{eq:condition}
          \| X-\E^s X \|_{L^p_\omega|\mathcal{F}_s} \le 2\| X-Y \|_{L^p_\omega|\mathcal{F}_s}\quad a.s.
      \end{equation}
 Let $f:[0,1]\times\Omega\rightarrow \mathbb R^d$ be a measurable function adapted to the filtration $(\mathcal F_t)_{t\ge0}$, $\gamma\in(0,1]$, $p\ge2$ and $[S,T]\subset[0,1]$. We give the following definitions:
        \begin{equation*}
            \begin{aligned}
                &\| f \|_{C_p^0[S,T]}\coloneqq\sup_{r \in [S,T]} \| f(r) \|_{L^p_\omega};\\
                &[f]_{C_p^{\gamma}[S,T]} \coloneqq \sup_{r_1, r_2 \in [S,T],r_1\neq r_2} \frac{\| \partial^{\lfloor \gamma \rfloor} f(r_1)-\partial^{\lfloor \gamma \rfloor}f(r_2) \|_{L^p_\omega}}{|r_1-r_2|^{\{\gamma\}}};\\
                &\| f \|_{C_p^{\gamma}[S,T]}\coloneqq\| f \|_{C_p^0[S,T]}+[f]_{C_p^{\gamma}[S,T]}.
            \end{aligned}
        \end{equation*}
    If $f$ is an adapted process, we choose $Y$ in~\eqref{eq:condition} as the value at $t$ of the Taylor expansion of $f$ at $s$ up to order $\lfloor \gamma \rfloor$ and we obtain
    \begin{equation}\label{eq:fregularity}
        \| f_t-\E^s f_t \|_{L^p_\omega} \le 2|t-s|^{\gamma}[f]_{{C}^{\gamma}_p[s,t]}.
    \end{equation}

In proofs, the notation $a\lesssim b$ (respectively $a\asymp b$) abbreviates the existence of $C>0$ such that $a\leq C b$ (respectively $C^{-1}b\leq a\leq C b$), such that moreover $C$ depends only on the parameters claimed in the corresponding statement.
\subsection{Main results}
By scaling we can always take $t\in[0,1]$ without lost of generality. 
Our main assumption and results are stated as follows. 
\begin{assumption}
    \label{ass:main1}
  Let $H\in (1,\infty)\backslash\mZ$, $b\in \cC^\alpha$ with $\alpha\in(1-\frac{1}{2H},1]$. 
\end{assumption}
Notice that following \cite{GG}, under the assumption above, there exists a unique strong solution to \eqref{eq:SDE}. Here is our numerical approximation result for it. 
\begin{theorem}
    \label{thm:main}
Let $(X_t)_{t\in[0,1]}, (X_t^n)_{t\in[0,1]}$ be the solutions to \eqref{eq:SDE} and \eqref{eq:SDE-EM} accordingly.  Suppose \cref{ass:main1} holds. Then for every $p\geq1$ and $\gamma<1+(\alpha-1)H$, we have
        \begin{align}
            \label{est:strong-main}
             \| X-X^n \|_{C_p^{\gamma}[0,1]} \le C |x_0-x_0^n| +C n^{-1},
        \end{align}
      where $C=C(p,d,\alpha,H,\|b\|_{\cC^\alpha})$.
\end{theorem}
\begin{remark}
    As an easy application of Kolmogorov continuity criteria we can also conclude from \eqref{est:strong-main} the following:
    \begin{align*}
       \big \|\sup_{t\in[0,1]}|X_t-X_t^n|\big\|_{L_\omega^p}\leq C |x_0-x_0^n| +C n^{-1}.
    \end{align*}
\end{remark}
The following theorem settles the question on optimality. 

\begin{theorem}
    \label{thm:main1} Let $(X_t)_{t\in[0,1]}, (X_t^n)_{t\in[0,1]}$ be the solutions to \eqref{eq:SDE} and \eqref{eq:SDE-EM} accordingly and $x_0=x_0^n$.  Suppose \cref{ass:main1} holds.  Then for every $f\in \cC^\alpha(\mR^d)$ there exists a $(\cF_t)_{t\in[0,1]}$-adapted process $(\cH f)^X_{t\in[0,1]}\in C_p^{1+(\alpha-1)H-}[0,1]$ for every $p\ge 2$ such that bounded linearly 
    \begin{align*}
       \cC^{\alpha}\ni f\mapsto  (\cH f)^X\in C_p^{1+(\alpha-1)H-}[0,1]
    \end{align*}
    and  if $g\in\cC^1(\mR^d)$, with probability one
         \begin{align*}
             (\cH g)_t^X=\int_0^t\nabla g(X_s)\mathrm{d} s,\quad t\in[0,1];
         \end{align*}
    moreover, we have for any $t\in[0,1]$, in probability
        \begin{align}
            \label{eq:optimal}
            \lim_{n\rightarrow\infty} n(X_t-X_t^n)=:c(t) 
        \end{align}
     exists and $c\in \cC^{1+(\alpha-1)H-}([0,1])$ satisfies
        \begin{equation}\label{opt-ODE}
            c(t)=  \int_0^tc(s)\mathrm{d} (\cH b)^X_s+ \frac 12 (b(X_t)-b(x_0)), \quad t\in[0,1],
        \end{equation} 
        where the integral inside \eqref{opt-ODE} is defined in the sense of Young.
\end{theorem}
\noindent{\bf Idea of the analysis.} The detailed proofs of \cref{thm:main} and \cref{thm:main1} will be given in \cref{sec:proof} and \cref{sec:optimal}, respectively. Here we only outline the main ideas.  
\begin{enumerate}
    \item 
(Convergence) 
Observe that for any $p\geq 1$, for $0\leq s \leq t\leq1$,
\begin{align*}
   \big \|(X_t-X_t^n)-(X_s-X_s^n)\big\|_{L^p_\omega}= \big\|(\varphi_t-\varphi_t^n)-(\varphi_s-\varphi_s^n)\big\|_{L^p_\omega},
\end{align*}
where 
\begin{align*}
    &\varphi_t:=X_t-B_t^H=\int_0^tb(X_r)\mathrm{d} r=\int_0^tb(\varphi_r+B_r^H)\mathrm{d} r,\\
    &\varphi_t^n:=X_t^n-B_t^H=\int_0^tb(X_{k_n(r)}^n)\mathrm{d} r=\int_0^tb(\varphi_{k_n(r)}^n+B_{k_n(r)}^H)\mathrm{d} r,
\end{align*}
meanwhile, we know that our aim is to remove the $\cC^1$ regularity requirement on $b$ via proper use of the regularization coming from $B^H$, namely the Gaussian density $p_{c(H)t^{2H}}$, which is infinitely smoothing (see \cref{subsec:fbm}). Also keep in mind that in the current setting the Girsanov Theorem is {\it not} available. What we have learned from \cite{GG, BDG-Levy} is that we can achieve this goal via {\it freezing} the exponent $\varphi_r$ inside the integral $\int_0^tb(\varphi_r+B_r^H)\mathrm{d} r$ (and similarly for $\varphi^n_{k_n(r)}$ inside the integral $\int_0^tb(\varphi^n_{k_n(r)}+B_r^H)\mathrm{d} r$)
by {\it taking conditional expectation} in the framework of the {\it SSL}. That is to say, heuristically, for $t-s$ small enough,
\begin{align*}
  \int_s^tb(\varphi_r+B_r^H)\mathrm{d} r &\overset{\|\cdot\|_{L^p_\omega}} {\approx}\int_s^t\mE^{s-(t-s)}b(\mE^{s-(t-s)}\varphi_r+B_r^H)\mathrm{d} r,\\
  \int_s^tb(\varphi^n_{k_n(r)}+B_r^H)\mathrm{d} r&\overset{\|\cdot\|_{L^p_\omega}}{\approx} \int_s^t\mE^{s-(t-s)}b(\mE^{s-(t-s)}\varphi^n_{k_n(r)}+B_r^H)\mathrm{d} r,
\end{align*}
and this ``$\overset{\|\cdot\|_{L^p_\omega}} {\approx}$'' is justified by the SSL (see \cref{lem:sewing-2} below) by taking
\begin{align*}
    A_{s,t}:=\int_s^t\mE^{s-(t-s)}b(\mE^{s-(t-s)}\varphi_r+B_r^H)\mathrm{d} r, \quad \cA_{s,t}:=\int_s^tb(\varphi_r+B_r^H)\mathrm{d} r,
\end{align*}
and similarly for $ \int_s^tb(\varphi^n_{k_n(r)}+B_r^H)\mathrm{d} r$. Then, together with the property of the Gaussian density $p_{c(H)t^{2H}}$ of fBM ( $\mathcal{P}_t^Hf:=p_{c(H)t^{2H}}\ast f$), we can further write 
\begin{align*}
    A_{s,t}= \int_s^t \mathcal{P}^H_{r-[s-(t-s)]} b (\E^{s-(t-s)} B^H_r + \E^{s-(t-s)} \varphi_r)\mathrm{d} r.
\end{align*}
Now we can see that instead of dealing with $b$ directly, we gain additional regularity in $\mathcal{P}^H_{t} b$ due to the smoothing effect of the convolution with $p_{c(H)t^{2H}}$. 

Although the full analysis later also contains many technical details, the core of the argument is clear. In the end, we are able to turn the idea above into a proof of the convergence rate in the following form:
\begin{align*}
    \big\|(X_t-X_t^n)-(X_s-X_s^n)\big\|_{L^p_\omega}=\big\|\varphi_t-\varphi_t^n\big\|_{L^p_\omega} &\leq  C ( \| \varphi-\varphi^n \|_{C^{\gamma}_p[s,t]} + n^{-1}) |t-s|^{\gamma+\varepsilon}\\
    &=C ( \| X-X^n \|_{C^{\gamma}_p[s,t]} + n^{-1}) |t-s|^{\gamma+\varepsilon}
\end{align*}
for sufficiently small $\varepsilon$. Therefore, we obtain \eqref{est:strong-main} after finely dividing the interval $[0,1]$ and applying the above estimate on each sub-interval.

\item (Optimality) Our idea for verifying optimality is to establish a limit theorem for the asymptotic error distribution of $X^n$ and $X$, that is, to show that $n(X^n-X)$ converges to a nontrivial limit, indicating that order $n$ is the optimal convergence rate of $X^n$ from \eqref{eq:SDE-EM} to $X$.   

For this, morally, if the coefficient $b$ is smooth, we can show that there exists a nonzero process $c(t)$ such that $c(t)=\lim_{n\rightarrow\infty} n(X_t-X_t^n)$ in probability
and it satisfies a random ODE
\begin{align}
    \label{eq:ODE1}c(t)=  \int_0^t\nabla b (X_{s})c(s)ds+ \frac 12 (b (X_{t})-b(x_0)).
\end{align}
At this point, we already see the difficulty in our setting: $b\in \cC^\alpha$ with $\alpha\in(1-\frac{1}{2H},1)$ implies that $\nabla b (X)$ is ill-defined in the above equation. 
        
However, if we reformulate \eqref{eq:ODE1} equivalently as follows: denote $H_t:=\int_0^t \nabla b(X_s)\mathrm{d} s$,
\begin{align}
    \label{eq:ODE2} c(t)= \int_0^tc(s)\mathrm{d} H_s+ \frac 12 (b (X_{t})-b(x_0)),\end{align}
then, due to the regularization effect, we can show that a.s. $H$ is $\cC^{1+(\alpha-1)H-}$-H\"older continuous (see \cref{lem:reg}), which implies that the term $\int_0^tc(s)\mathrm{d} H_s$ can be defined in the sense of the Young integral. Therefore, we can claim that \eqref{eq:ODE1} is well defined, given that \eqref{eq:ODE2} is a linear equation. In this way, we also show that $n(X^n-X)$ indeed converges to a nontrivial limit and thus that order $n$ is optimal. 

Let us also point out that the general idea of this part is close to that in \cite{DGL-CLT}. However, here we are dealing with additive fBM noise, whereas \cite{DGL-CLT} considers central limit theorem type results that heavily rely on the multiplicative structure of the noise; therefore, the detailed and essential analysis deviates. In the end, unlike the weakly convergence result from \cite{DGL-CLT}, we are able to show that $n(X^n_t-X_t)$  converges to $c(t)$ in probability.
\end{enumerate}

\section{Auxiliary Tools}\label{sec:auxi-tool}
In this section, we primarily introduce our main tool which is the stochastic sewing lemma and present some properties of the fractional Brownian motions.
\subsection{Stochastic Sewing Lemma}  Given $ M \geq 0$ we define $[S,T]_M^2=\{ (s,t)|S \le s < t \le T, s-M(t-s)\ge S \}$ and $\overline{[S,T]}_M^3=\{ (s,u,t)|(s,t) \in [S,T]_M^2, (u-s)\wedge(t-u)\ge \frac{t-s}{3}\}$.
\begin{lemma}\cite[Lemma {2.2}]{GG} \label{lem:sewing-2}
    Let $0 \leq S<T \leq 1, p \in[2, \infty), M \geq 0$ and let $\left(A_{s, t}\right)_{(s, t) \in[S, T]_{M}^{2}}$ be a family of random variables in $L^{p}\left(\Omega, \mathbb{R}^{d}\right)$ such that $A_{s, t}$ is $\mathcal{F}_{t}$-measurable. Suppose that for some $\varepsilon_{1}, \varepsilon_{2}>0$ and $C_{1}, C_{2}$ the bounds
        \begin{equation}\label{sew:A}
            \big\|A_{s, t}\big\|_{L^{p}_{\omega}}  \leq C_{1}|t-s|^{1 / 2+\varepsilon_{1}}
        \end{equation}
    and
        \begin{equation}\label{sew:deltaA}
            \big\|\mathbf{E}^{s-M(t-s)} \delta A_{s, u, t}\big\|_{L^{p}_{\omega}}  \leq C_{2}|t-s|^{1+\varepsilon_{2}}
        \end{equation}
    hold for all $(s, t) \in[S, T]_{M}^{2}$ and $(s, u, t) \in \overline{[S, T]}_{M}^{3}$, where $\delta A_{s,u,t}:=A_{s,t}-A_{s,u}-A_{u,t}$. Then there exists a unique (up to modification) adapted process $\mathcal{A}:[S, T] \rightarrow L^{p}\left(\Omega, \mathbb{R}^{d}\right)$ such that $\mathcal{A}_{S}=0$ and such that for some constants $K_{1}, K_{2}<\infty$, depending only on $\varepsilon_{1}, \varepsilon_{2}, p, d$, and $M$, the bound
        \begin{equation}\label{eq:sew:A_bound}
            \left\|\mathcal{A}_{t}-\mathcal{A}_{s}\right\|_{L^{p}_{\omega}} \leq K_{1} C_{1}|t-s|^{1 / 2+\varepsilon_{1}}+K_{2} C_{2}|t-s|^{1+\varepsilon_{2}}
        \end{equation}
    holds for all $(s, t) \in[S, T]_{0}^{2}$. Moreover, if there exists any continuous process $\widetilde{\mathcal{A}}$ : $[S, T] \rightarrow L^{p}\left(\Omega, \mathbb{R}^{d}\right), \varepsilon_{3}>0$, and $K_{3}<\infty$, such that $\widetilde{\mathcal{A}}_{S}=0$ and
        \begin{equation} \label{sew:A-A}
            \big\|\widetilde{\mathcal{A}}_{t}-\widetilde{\mathcal{A}}_{s}-A_{s, t}\big\|_{L^{p}_{\omega}} \leq K_{3}|t-s|^{1+\varepsilon_{3}}
        \end{equation}
    holds for all $(s, t) \in[S, T]_{M}^{2}$, then $\widetilde{\mathcal{A}}_{t}=\mathcal{A}_{t}$ for all $S \leq t \leq T$.
\end{lemma}
\subsection{Fractional Brownian motions}\label{subsec:fbm}
We recall the following properties concerning $B^H$   that have been used heavily in later analysis.

\begin{lemma}\cite[Proposition {2.1}]{GG}
    For any $H \in (0,\infty) \textbackslash \mathbb{Z}$ there exists a constant $c(H)$ such that for all $0 \le s \le t \le 1$ one has 
        \begin{align}\label{eq:rep1}
            \E |B_t^H-\E^s B_t^H |^2=dc(H) |t-s|^{2H} \text{ and } B_t^H-\E^s B_t^H \text{ is independent of } \mathcal F_s.
        \end{align}
\end{lemma}
    For any $H \in (0,\infty) \textbackslash \mathbb{Z}$ there exists a constant $C=C(d,H)$ such that for all $0 \le s \le t \le 1$ one has 
        \begin{align}\label{eq:rep}
            \E |B_t^H- B_s^H | \le C |t-s|^{H \wedge 1}.
        \end{align}
        
    We let $p_t(x)$ denote the known heat density $\frac{1}{(2\pi t)^{d/2}}e^{-\frac{|x|^2}{2t}}$ on $\mR^d$ and we define $\cP_t^Hf(x):=(p_{c(H)t^{2H}}\ast f)(x)$, $x\in\mR^d$. Then for any $\mathcal F_s$-measurable $\mR^d$ valued random vector $\xi$, we have
        \begin{equation}\label{eq:EBt}
            \mE^sf(B_t^H+\xi)=\cP_{t-s}^H f(\mE^sB_t^H+\xi).
        \end{equation}

\begin{lemma}\label{heat kernel}
    For $\alpha,\beta \in [0,1],f\in \cC^\alpha, t \in (0,1],$ one has the bounds, with some constant $C$ depending only on $H, \alpha, \beta, d$, accordingly
      \begin{align}
            \label{eq:heat2}
           & |\mathcal{P}_t^H f(x_1)-\mathcal{P}_t^H f(x_2)-\mathcal{P}_t^H f(x_3)+ \mathcal{P}_t^H f(x_4)|
          \notag\\
            &   \le C \|f\|_{\cC^\alpha}\big( t^{H(\alpha-2)} |x_1-x_2||x_1-x_3|  + t^{H(\alpha-1)} |x_1-x_2-x_3+x_4| \big),\forall x_i\in\mR^d, i=1,\ldots,4;
            \\
             \label{est:heat-1}
            &\|\mathcal{P}_t^H f  \|_{\cC^\beta} \le C t^{H(\alpha-\beta)\wedge0}\|f\|_{\cC^\alpha};\\
            \label{est:heat-2}
            & \|(\mathcal{P}_{t}^H-\mathcal{P}_{s}^H )f  \|_{\cC^\beta} \le C s^{H(\alpha-\beta)-2H\delta} |t^{2H}-s^{2H}|^{\delta}\|f\|_{\cC^\alpha}, \forall 0\leq s\leq t\leq 1,  0<\delta\in \left[\frac{\alpha-\beta}{2},1 \right].
        \end{align} 
\end{lemma}
\begin{proof}
     \eqref{eq:heat2} are directly from \cite[Page 2 (2.8)]{GG}. By using properties of Gaussian convolutions, heat kernel bounds and a relation of the
form $\cP_t^Hf(x)=(p_{c(H)t^{2H}}\ast f)(x)$, we get \eqref{est:heat-1} from \cite[Proposition 3.7 (i)]{BDG}.  
For \eqref{est:heat-2}, it holds  from \cite[Proposition 3.7 (ii)]{BDG}.
\end{proof}

\section{Strong convergence rate}\label{sec:proof}
In this part we give the proof for \cref{thm:main}. 

Denote
\begin{align*}
    \varphi_t\coloneqq(X-B^H)_t&=x_0+\int_0^t b(\varphi_s+B^H_s) \mathrm{d}s,\\
    \varphi^n_t\coloneqq(X^n-B^H)_t&=x_0^n+\int_0^t b(\varphi_{k_n(s)}^n+B^H_{k_n(s)}) \mathrm{d} s.
\end{align*}
Fix $S \le s < t \le T$ and $[S,T] \subset [0,1]$. 
    We write
        \begin{equation*}
            \begin{aligned}
                (X&-X^n)_t-(X-X^n)_s\\&=(\varphi-\varphi^n)_t-(\varphi-\varphi^n)_s \\
                &=\int_s^t b(B^H_r+\varphi_r)-b(B^H_r+\varphi_r^n) \mathrm{d} r + \int_s^t b(B^H_r+ \varphi_r^n)-b(B^H_r+ \varphi_{k_n(r)}^n ) \mathrm{d}r \\
                & \quad + \int_s^t b(B^H_r+ \varphi^n_{k_n(r)})-b(B^H_{k_n(r)}+\varphi_{k_n(r)}^n) \mathrm{d} r\\
            &\eqqcolon \mathcal E^{b,n,1}_{s,t}+\mathcal E^{b,n,2}_{s,t}+\mathcal E^{b,n,3}_{s,t}.
        \end{aligned}
    \end{equation*}
   It is clear that in order to show \eqref{est:strong-main}, we need to estimate $\cE^{b,n,1}_{s,t}, \cE^{b,n,2}_{s,t}, \cE^{b,n,3}_{s,t}$ individually. 
We distribute the estimates for each into  \cref{lem:E1}, \cref{lem:E2} and \cref{lem:E3} correspondingly.

Before that 
we first present the following auxiliary lemma for the processes $\varphi$ and $\varphi^n$ defined above which will be heavily used in the later proofs.
\begin{lemma}
    Assume \cref{ass:main1} holds. Then for all $t>s$ and $p\geq 1$ we have a.s.
    \begin{align} \label{eq:phi}
        \| \varphi_t-\E^s \varphi_t \|_{L^p_{\omega}|\mathcal{F}_s} & \le C \|b\|_{\cC^\alpha}|t-s|^{1+\alpha H}; \\ \label{eq:phi-1}
        \| \varphi^n_t -\E^s \varphi_t^n \|_{L^p_\omega| \mathcal{F}_s} & \le C \|b\|_{\cC^\alpha}|t-s|^{1+\alpha H}
    \end{align}
    with some constant $C=C(p,d,\alpha,H)$. 
\end{lemma}
\begin{proof}
    We only give the proof of~\eqref{eq:phi-1}. 

    For fixed $s,$ define $s'$ to be the smallest grid point which is bigger or equal to $s,$ that is, $s':=\lceil ns \rceil n^{-1}$. It is crucial to note that $\varphi_{s'}^n$ is $\mathcal{F}_s$-measurable.
    Suppose \eqref{eq:phi-1} holds for some $m \ge 0$ in place of $1+\alpha H.$  This is certainly true for $m=0$ thanks to the fact that $b$ is bounded; we proceed now by induction on $m.$ If $s \le t <s',$ then $\varphi_t^n$ is $\mathcal{F}_s$-measurable. Hence $\varphi^n_t=\E^s \varphi_t^n$ and the left-hand side of \eqref{eq:phi-1} is zero. Therefore it remains to consider the case $t \ge s'.$ In this case, using  \eqref{eq:condition} with $X=\varphi_t^n, Y=\varphi^n_{s'}+\int_{s'}^{t} b(\E^s B_{k_n(r)}^H+\E^s \varphi_{k_n(r)}^n ) \mathrm{d} r,$ we deduce
    \begin{align*}
        \| \varphi_t^n-\E^s \varphi_t^n \|_{L^p_\omega|\mathcal{F}_s}
        & \le 2 \left\| \varphi_t^n - \varphi^n_{s'}-\int_{s'}^{t} b (\E^s B^H_{k_n(r)}+ \E^s \varphi_{k_n(r)}^n ) \mathrm{d} r \right\|_{L^p_\omega|\mathcal{F}_s} \\
        & =2 \left\| \int_{s'}^t ( b(B^H_{k_n(r)}+\varphi_{k_n(r)}^n)-b(\E^s B_{k_n(r)}^H+\E^s \varphi_{k_n(r)}^n ) \mathrm{d} r  \right\|_{L^p_\omega|\mathcal{F}_s} \\
        & \le C \|b\|_{\cC^\alpha}\left\| \int_{s'}^t (|B^H_{k_n(r)}-\E^s B^H_{k_n(r)}|^\alpha + | \varphi_{k_n(r)}^n-\E^s \varphi^n_{k_n(r)}|^\alpha )  \mathrm{d} r \right\|_{L^p_\omega|\mathcal{F}_s}.
    \end{align*}
    Using \eqref{eq:EBt} and the induction hypothesis, we get
    \begin{equation*}
        \| \varphi_t - \E \varphi_t \|_{L^p_\omega|\mathcal{F}_s} \le C\|b\|_{\cC^\alpha} |t-s|^{(H\alpha) \wedge (m \alpha)+1}\quad a.s.
    \end{equation*}
    We note that $m_0=0, m_{i+1}=1+(H \alpha) \wedge (m_{i} \alpha)$ reaches $1+H \alpha$ in finitely many steps, therefore we get \eqref{eq:phi}.
\end{proof}

    Here we introduce some notations commonly used in the proofs of ~\cref{lem:E1},~\cref{lem:E2} and~\cref{lem:E3}. For $(s,u,t) \in\overline{[0, 1]}_{1}^{3}$, we set 
        \begin{equation}\label{eq:points}
            s_1\coloneqq s-(t-s), s_2\coloneqq u-(t-u), s_3\coloneqq s-(u-s), s_4\coloneqq s, s_5\coloneqq u, s_6\coloneqq t.
        \end{equation}
    Note by the fact $u \le \frac{2}{3} s+ \frac{1}{3} t$ for $(s,u, t) \in \overline{[0, 1]}_{1}^{3}$, we have $s_2 \le s_3.$

Let us start with the estimate for $\mathcal E^{f,n,1}$.
\begin{lemma}\label{lem:E1}
   Suppose \cref{ass:main1} holds. Then  for any $p\geq 1$ and $\gamma\in(\frac12,1+(\alpha-1)H)$ we have
         \begin{align}\label{est:E1}
            &\left\| \mathcal E^{f,n,1}_{s,t} \right\|_{L^p_\omega} \le C\|f\|_{\cC^\alpha}\| \varphi-\varphi^n \|_{C_p^{\gamma}[s,t]} |t-s|^{\gamma+\varepsilon}, \quad (s,t)\in[0,1]_0^2,f\in \cC^{\alpha}
        \end{align}
    with sufficiently small $\eps>0$ and some constant $C=C(p,d,\alpha,\gamma,H,\eps,\|b\|_{\cC^{\alpha}})$.
\end{lemma}

\begin{proof} The idea is to apply \cref{lem:sewing-2}. 
   Let  $M=1$, $(s,t)\in[0,1]_1^2$ and  $$A_{s,t}:=\E^{s-(t-s)} \int_s^t f (B^H_r+ \E^{s-(t-s)} \varphi_r)-f(B^H_r+\E^{s-(t-s)} \varphi_r^n) \mathrm{d}r. $$
    We are going to verify \eqref{sew:A} and \eqref{sew:deltaA}. 
    By~\eqref{eq:EBt}, we see
     \begin{align}\label{eq:A_st_heat}
        A_{s,t} & = \int_s^t \mathcal{P}^H_{r-[s-(t-s)]} f (\E^{s-(t-s)} B^H_r + \E^{s-(t-s)} \varphi_r) \notag \\
        & \qquad -\mathcal{P}^H_{r-[s-(t-s)]} f(\E^{s-(t-s)} B^H_r+\E^{s-(t-s)}\varphi^n_r) \mathrm{d} r.
     \end{align}
     Then by ~\eqref{est:heat-1} and~\eqref{eq:CJI}, we get
        \begin{equation}\label{eq:AstLp}
            \begin{aligned}
                \| A_{s,t} \|_{L^p_\omega} &\le C \|f\|_{\cC^\alpha} \int_s^t (r-[s-(t-s)])^{-(1-\alpha)H} \| \E^{s-(t-s)} (\varphi_r-\varphi_r^n) \|_{L^p_\omega} \mathrm{d}r\\
                & \le C \|f\|_{\cC^\alpha} \| \varphi-\varphi^n \|_{C_p^0[s,t]} |t-s|^{1-(1-\alpha)H}.
            \end{aligned}
        \end{equation}
    Then~\eqref{sew:A} holds  with $C_1=C \|f\|_{\cC^\alpha[s,t]} \| \varphi-\varphi^n \|_{C_p^0}$ by the fact that $1-(1-\alpha)H > \frac{1}{2}$.
    
    Next we verify~\eqref{sew:deltaA}. Let $(s,u,t)\in\overline{[0,1]}_1^3$. Recall the definition of $s_i,i=1,\dots,6$ in~\eqref{eq:points}. We first can write
        \begin{equation}
             \begin{aligned}
            &\E^{s-(t-s)} \delta A_{s,u,t}\\ 
            =&\E^{s_1} \E^{s_3} \int_{s_4}^{s_5}  f(B_r^H+ \E^{s_1} \varphi_r)- f( B_r^H + \E^{s_1} \varphi_r^n)-f(B^H_r + \E^{s_3} \varphi_r) +f( B^H_r + \E^{s_3} \varphi_r^n) \mathrm{d} r \\
            &+ \E^{s_1} \E^{s_2} \int_{s_5}^{s_6} f( B_r^H+\E^{s_1} \varphi_r) - f( B_r^H + \E^{s_1} \varphi_r^n)- f( B_r^H+\E^{s_2} \varphi_r)+f( B_r^H+\E^{s_2} \varphi_r^n) \mathrm{d} r \\
            =&\E^{s_1} \int_{s_4}^{s_5} \mathcal{P}^H_{r-s_3} f(\E^{s_3} B_r^H+ \E^{s_1} \varphi_r)-\mathcal{P}^H_{r-s_3} f(\E^{s_3} B_r^H+\E^{s_1} \varphi_r^n) \\
            &\qquad\qquad- \mathcal{P}^H_{r-s_3} f(\E^{s_3}B^H_r + \E^{s_3} \varphi_r)+\mathcal{P}^H_{r-s_3} f(\E^{s_3} B^H_r + \E^{s_3} \varphi_r^n)\mathrm{d} r \\
            &  + \E^{s_1} \int_{s_5}^{s_6} \mathcal{P}_{r-s_2}^H f(\E^{s_2} B_r^H+\E^{s_1} \varphi_r) - \mathcal{P}^H_{r-s_2} f(\E^{s_2} B_r^H + \E^{s_1} \varphi_r^n)\\
            &\qquad\qquad- \mathcal{P}^H_{r-s_2} f(\E^{s_2} B_r^H+\E^{s_2} \varphi_r)+ \mathcal{P}^H_{r-s_3}f(\E^{s_2} B_r^H+\E^{s_2} \varphi_r^n)\mathrm{d} r \\
            =: & I_1+I_2.\label{def:deltaA}
        \end{aligned}
    \end{equation}
    The two terms are treated in  the exactly same way, so we only detail $I_1$.
    By~\eqref{eq:CJI} and applying~\eqref{eq:heat2} with 
    \begin{align*}
       x_1=\E^{s_3} B_r^H+\E^{s_1} \varphi_r^n, x_2=\E^{s_3} B_r^H+ \E^{s_1} \varphi_r, x_3=\E^{s_3} B^H_r + \E^{s_3} \varphi_r^n, x_4=\E^{s_3}B^H_r + \E^{s_3} \varphi_r, 
    \end{align*} we obtain
         \begin{equation}\label{eq:I_1_bound}
            \begin{aligned}
                \| I_1 \|_{L^p_\omega} & \le C \|f\|_{\cC^\alpha} \int_{s_4}^{s_5}  |r-s_3|^{-(1-\alpha)H}   \big\| \E^{s_1} | \E^{s_3} \varphi_r -\E^{s_3} \varphi_r^n-\E^{s_1} \varphi_r+\E^{s_1} \varphi_r^n | \big\|_{L^p_\omega} \\
                & \quad \qquad\qquad\quad+  |r-s_3|^{-(2-\alpha)H} \big\| |\E^{s_1} \varphi_r-\E^{s_1} \varphi_r^n | \cdot \E^{s_1} | \E^{s_3} \varphi_r^n-\E^{s_1} \varphi_r^n | \big\|_{L^p_\omega} \mathrm{d}r.
            \end{aligned}
         \end{equation}
   By~\eqref{eq:phi-1},
    \begin{equation}\label{eq:phi_n_reg}
        \E^{s_1} | \E^{s_3} \varphi_r^n -\E^{s_1} \varphi_r^n | 
        = \E^{s_1} | \E^{s_3} ( \varphi^n_r-\E^{s_1} \varphi_r^n)|  \le \E^{s_1} | \varphi^n_r-\E^{s_1} \varphi_r^n| \le C |r-s_1|^{1+\alpha H}.
     \end{equation}
   Besides from ~\eqref{eq:CJI}, we get
    \begin{equation}\label{eq:E1E3=E1}
        \begin{aligned}
            \big\|\E^{s_1} |\E^{s_3} \varphi_r-\E^{s_3} \varphi_r^n-\E^{s_1} \varphi_r+\E^{s_1} \varphi_r^n | \big\|_{L^p_\omega}
            & =\big\|\E^{s_1} | \E^{s_3} ((\varphi_r-\varphi_r^n)-\E^{s_1}(\varphi_r-\varphi_r^n))|\big\|_{L^p_\omega} \\
            & \le \big\|\E^{s_1} |(\varphi_r-\varphi_r^n)-\E^{s_1}(\varphi_r-\varphi_r^n)|\big\|_{L^p_\omega}\\
            & \le \| (\varphi_r-\varphi_r^n)-\E^{s_1} (\varphi_r-\varphi_r^n) \|_{L^p_\omega},
        \end{aligned}
    \end{equation}     
    meanwhile ~\eqref{eq:fregularity} implies
         \begin{equation}\label{eq:phi-phi_n_reg}
            \big\|\E^{s_1} |\E^{s_3} \varphi_r-\E^{s_3} \varphi_r^n-\E^{s_1} \varphi_r+\E^{s_1} \varphi_r^n | \big\|_{L^p_\omega} \le C |r-s_1|^{\gamma} [\varphi-\varphi^n]_{C^{\gamma}_p[s,t]},
        \end{equation}
 clearly from ~\eqref{eq:CJI}, 
       \begin{equation}\label{eq:phi-phi_n_bound}
        \| \E^{s_1} (\varphi_r-\varphi_r^n ) \|_{L^p_\omega} \le  \| \varphi-\varphi^n \|_{C_p^0[s,t]},
    \end{equation}
   now plugging ~\eqref{eq:phi_n_reg},~\eqref{eq:phi-phi_n_reg} and~\eqref{eq:phi-phi_n_bound} into~\eqref{eq:I_1_bound}, we have
        \begin{equation}\label{est:I1}
            \begin{aligned}
                \| I_1 \|_{L^p_\omega} & \le C \|f \|_{\cC^\alpha} [\varphi-\varphi^n]_{C_p^{\gamma}[s,t]} \int_{s_4}^{s_5} (r-s_1)^{\gamma} (r-s_3)^{-(1-\alpha)H} \mathrm{d} r \\
                & \quad + \|f\|_{\cC^\alpha} \| \varphi-\varphi^n \|_{C_p^0[s,t]} \int_{s_4}^{s_5} (r-s_1)^{1+\alpha H} (r-s_3)^{-(2-\alpha)H} \mathrm{d} r  \\
                & \le C \|f\|_{\cC^\alpha} (t-s)^{1+\gamma-(1-\alpha)H} [\varphi-\varphi^n]_{C_p^{\gamma}[s,t]}  \\
                & \quad + \|b \|_{\cC^\alpha} (t-s)^{2+\alpha H-(2-\alpha)H} \| \varphi-\varphi^n \|_{C_p^0[s,t]}\\
                & \le C \|f\|_{\cC^\alpha}  \| \varphi-\varphi^n \|_{C_p^{\gamma}[s,t]} (t-s)^{(1+\gamma-(1-\alpha)H)\wedge(2+\alpha H-(2-\alpha)H)}.
            \end{aligned}
        \end{equation}
   
 The above analysis also implies the same bound  on $I_2$ observing $I_1$ and $I_2$ share the same structure.   
  
  Noticing  ~\cref{ass:main1} implies $(1+\gamma-(1-\alpha)H)\wedge(2+\alpha H-(2-\alpha)H)>1$, we conclude ~\eqref{sew:deltaA} holds with $C_2= C \|f\|_{\cC^\alpha}  \| \varphi-\varphi^n \|_{C_p^{\gamma}[s,t]}$.
    
   Now we claim that the process $\mathcal{A}$ in~\eqref{eq:sew:A_bound} actually is given by 
        \begin{equation}
            {\mathcal{A}}_t= \int_0^t f(B^H_r+\varphi_r)-f(B^H_r+\varphi_r^n) \mathrm{d} r.
        \end{equation}
    To prove this, it suffices to show~\eqref{sew:A-A}. 
    By~\eqref{eq:A_st_heat}, we write
        \begin{equation}\label{con:uniq}
            \begin{aligned}
                {\mathcal{A}}_t-{\mathcal{A}}_s-A_{s,t} &=\int_s^t f(B^H_r+\varphi_r)-\mathcal{P}^H_{r-s_1} f(\E^{s_1} B_r^H + \E^{s_1} \varphi_r) \mathrm{d}r  \\
                &\quad  -\int_s^t f(B_r^H+ \varphi_r^n)-\mathcal{P}^H_{r-s_1} f(\E^{s_1} B_r^H + \E^{s_1} \varphi_r^n) \mathrm{d}r \\
                & =:II_1+II_2.
            \end{aligned} 
        \end{equation}
Again we can see that    $II_2$ can be treated similarly to $II_1$, so we only detail $II_1$. We can see
     \begin{align}\label{est:pt-12}
        II_1&=\int_s^t (f-\mathcal{P}^H_{r-s_1} f)(B_r^H+\varphi_r) \mathrm{d}r + \int_s^t \mathcal{P}^H_{r-s_1} f(B_r^H+\varphi_r)-\mathcal{P}^H_{r-s_1} f(\E^{s_1} B_r^H+\E^{s_1} \varphi_r) \mathrm{d}r.
     \end{align}

   Using~\eqref{est:heat-2} with $\delta=\frac{\alpha}{2}$, $\beta=0$ and~\eqref{est:heat-1} with $\beta=\alpha$, we get 
     \begin{align}
        \| II_1 \|_{L^p_\omega} 
        &\le C \int_s^t \| \mathcal{P}^H_{r-s_1}f-f \|_{\cC^0} + \| \mathcal{P}^H_{r-s_1} f \|_{\cC^\alpha} ( \|B_r^H-\E^{s_1} B^H_r \|^{\alpha}_{L^{\alpha p}_{\omega}}+ \| \varphi_r-\E^{s_1} \varphi_r\|^{\alpha}_{L_{\omega}^{\alpha p}}) \mathrm{d} r \nonumber\\
        & \le C \|f\|_{\cC^\alpha} \int_s^t (r-s_1)^{\alpha H}  + (|r-s_1|^{\alpha H}+|r-s_1|^{\alpha(1+\alpha H)}) \mathrm{d} r \nonumber\\
        & \le  C \|f \|_{\cC^\alpha} (t-s)^{(1+\alpha H) \wedge (1+\alpha (1+\alpha H))},\label{con:uni-1}
     \end{align}
     where in the second inequality we used~\eqref{eq:rep} and~\eqref{eq:phi}. The same bound on $II_2$.

    Therefore, ~\eqref{sew:A-A} holds since $(1+\alpha H) \wedge (1+\alpha (1+\alpha H))>1$. Then the uniqueness from \cref{lem:sewing-2} verifies the claim.
   
     Finally, by \cref{lem:sewing-2}, the proof completes.
\end{proof}

Let us move to estimate $\mathcal E^{f,n,2}$ term.
\begin{lemma}\label{lem:E2}
   Suppose \cref{ass:main1} holds. Then we have for any $p\geq 1$, $\gamma\in(\frac12,1+(\alpha-1)H)$ and $(s,t)\in[0,1]_0^2$  
        \begin{equation}\label{est:E2}
            \left\| \mathcal E^{f,n,2}_{s,t} \right\|_{L^p_\omega} \le \frac{C}{n} \|f\|_{\cC^\alpha}|t-s|^{\gamma+ \varepsilon},\quad (s,t)\in[0,1]_0^2,f\in \cC^{\alpha}
        \end{equation}
    where $\epsilon>0$ is sufficiently small and  $C=C(p,d,\alpha,\gamma,H,\eps,\|b\|_{\cC^{\alpha}})$.
\end{lemma}
    
\begin{proof}
    Again the idea is to apply \cref{lem:sewing-2}.    Let  $M=1$, $(s,t)\in[0,1]_1^2$ and \begin{equation}\label{eq:A_st2}
        A_{s,t}:=\E^{s-(t-s)} \int_s^t f(B_r^H+ \E^{s-(t-s)} \varphi_r^n)-f(B_r^H+\E^{s-(t-s)} \varphi_{k_n(r)}^n) \mathrm{d}r.
    \end{equation}
We may notice that the analysis of showing \eqref{eq:AstLp} and estimating \eqref{def:deltaA} also work for \eqref{eq:A_st2}. Therefore we omit the detailed proof here for showing the following estimates and present in \cref{app:E2};
in the end we get 
    \begin{align}
        \label{est:E2-app-1}
     \| A_{s,t} \|_{L^p_\omega}&\le C \|f\|_{\cC^\alpha} \sup_{r \in [s,t]} \| \varphi_r^n-\varphi_{k_n(r)}^n \|_{L_\omega^p} |t-s|^{1-(1-\alpha)H}
           \nonumber\\& \le C n^{-1} \|f\|_{\cC^\alpha} |t-s|^{1-(1-\alpha)H},\\
     \|  \E^{s_1} \delta A_{s,u,t}\|_{L^p_\omega}&\le C \frac{\|f\|_{\cC^\alpha}}{n} (t-s)^{(1+\alpha-(1-\alpha)H) \wedge (2+(2 \alpha -2)H)},
     \label{est:E2-app-2}
       \end{align}
 recall the definition of $s_i,i=1,\dots,6$ in~\eqref{eq:points}.

  Then, with taking $C_1=C\|f\|_{\cC^\alpha}n^{-1}$, \eqref{sew:A} holds by the fact that $1-(1-\alpha)H > \frac{1}{2}.$   ~\cref{ass:main1} implying  $(1+\alpha-(1-\alpha)H) \wedge (2+(2 \alpha -2)H)>1$, ~\eqref{sew:deltaA} holds with $C_2= \frac{C}{n} \|f\|_{\cC^\alpha}$

   Similarly to~\cref{lem:E1}, we could verify that 
      the process $\mathcal{A}$ in~\eqref{eq:sew:A_bound} is indeed given by  
        \begin{equation}
            \mathcal{A}_t= \int_0^t f(B^H_r+\varphi_r^n)-f(B^H_r+\varphi_{k_n(r)}^n) \mathrm{d} r.
        \end{equation}
    
    Therefore, the claim \eqref{est:E2} holds by the uniqueness from \cref{lem:sewing-2}. 
\end{proof}
Now it is the analysis for the last term--$\mathcal E^{f,n,3}$.
\begin{lemma}\label{lem:E3}
     Suppose \cref{ass:main1} holds. Then for any $p\geq 1$, $\gamma\in(\frac12,1+(\alpha-1)H)$ and $(s,t)\in[0,1]_0^2$, we have 
        \begin{equation}\label{est:E3}
            \left\| \mathcal E^{f,n,3}_{s,t}\right\|_{L^p_\omega} \le \frac{C}{n} \|f\|_{\cC^\alpha}|t-s|^{\gamma+ \varepsilon},\quad (s,t)\in[0,1]_0^2,f\in \cC^{\alpha}
        \end{equation}
    with sufficiently small $\eps>0$ and some constant $C=C(p,d,\alpha,\gamma,H,\eps,\|b\|_{\cC^\alpha})$.
\end{lemma}
\begin{proof}
    In order to apply \cref{lem:sewing-2}, this time  we set $$A_{s,t}=\E^{s-(t-s)} \int_s^t f (B_r^H+ \E^{s-(t-s)} \varphi_{k_n(r)}^n)-f(B_{k_n(r)}^H+\E^{s-(t-s)} \varphi_{k_n(r)}^n) \mathrm{d}r.$$ 
    When $|t-s| \le \frac{10}{n}$, by~\eqref{eq:rep}, we have for any $\epsilon>0$
        \begin{align*}
            \|A_{s,t} \|_{L^p_\omega} 
            & \le C \|f\|_{\cC^\alpha} \int_s^t \big\| |B^H_r-B^H_{k_n(r)} |^\alpha \big\|_{L^p_\omega} \mathrm{d} r \le C\|f\|_{\cC^\alpha}  |t-s| \cdot \frac{1}{n^\alpha}\\& \le C \|f\|_{\cC^\alpha} |t-s|^{\gamma+\varepsilon} \frac{1}{n^{\alpha + 1-\gamma-\varepsilon}}. 
        \end{align*}
    Since $\alpha>\gamma$ we can take such sufficiently small $\eps>0$ so that $\epsilon\in(0,\alpha-\gamma)$, which implies
    \begin{equation}\label{eq:A_st:t-ssmall}
        \|A_{s,t} \|_{L^p_\omega}  \le \frac{C}{n} \|f\|_{\cC^\alpha}  |t-s|^{\frac{1}{2}+\varepsilon}. 
    \end{equation}
    When $|t-s| > \frac{10}{n}$, we have
       \begin{align*}
           A_{s,t} & =\int_{s_4}^{s_6} \mathcal{P}^H_{r-s_1} f(\E^{s_1} B_r^H+ \E^{s_1} \varphi_{k_n(r)}^n) -\mathcal{P}^H_{k_n(r)-s_1} f(\E^{s_1} B_{k_n(r)}^H+ \E^{s_1} \varphi_{k_n(r)}^n) \mathrm{d}r \\
           &=\int_{s_4}^{s_6} \mathcal{P}^H_{r-s_1} f(\E^{s_1} B_r^H+\E^{s_1} \varphi_{k_n(r)}^n)-\mathcal{P}^H_{r-s_1} f(\E^{s_1} B_{k_n(r)}^H + \E^{s_1} \varphi_{k_n(r)}^n) \\
           & \qquad +\mathcal{P}^H_{r-s_1} f(\E^{s_1} B_{k_n(r)}^H+ \E^{s_1} \varphi_{k_n(r)}^n)-\mathcal{P}^H_{k_n(r)-s_1} f(\E^{s_1} B^H_{k_n(r)}+\E^{s_1} \varphi_{k_n(r)}^n) \mathrm{d}r \\
           &=:IV_1+IV_2.
       \end{align*}
       For $IV_1$, by \eqref{est:heat-1}, \eqref{eq:CJI} and \eqref{eq:rep}, we have
       \begin{align}
           \| IV_1 \|_{L^p_\omega} &\le \int_{s_4}^{s_6} \| \mathcal{P}^H_{r-s_1} f \|_{\cC^1} \| \E^{s_1} (B_r^H-B_{k_n(r)}^H)\|_{L^p_\omega} \mathrm{d}r
            \le \frac{C\|f\|_{\cC^\alpha}}{n}  \int_{s_4}^{s_6} (r-s_1)^{-(1-\alpha)H}  \mathrm{d}r \nonumber\\
           & \le \frac{C}{n} \|f\|_{\cC^\alpha} |t-s|^{1-(1-\alpha)H}.\label{est:I1-3}
       \end{align}
       For $ IV_2$,  \eqref{est:heat-2} with $\delta=1$ gives us
    $$\| \mathcal{P}^H_t f-\mathcal{P}_s^H f\|_{\cC^0} \le C s^{-(2-\alpha)H} |t^{2H}-s^{2H}| \|f\|_{\cC^\alpha},$$
  it implies 
       \begin{align*}
           |IV_2| & \le C \|f \|_{\cC^\alpha} \int_{s_4}^{s_6} (k_n(r)-s_1)^{-(2-\alpha)H} \left( (r-s_1)^{2H} -(k_n(r)-s_1)^{2H} \right) \mathrm{d} r .
       \end{align*}
    Moreover, by $k_n(r)-s_1 \asymp t-s,$ and $$\big| |r-s_1|^{2H}-|k_n(r)-s_1|^{2H}\big| \le C |r-k_n(r)| |r-s_1|^{2H-1} \le \frac{C}{n}  |t-s|^{2H-1},$$ we have
       \begin{equation}
           |IV_2| \le  \frac{C}{n} \|f\|_{\cC^\alpha} \int_{s_4}^{s_6} (t-s)^{2H-1-(2-\alpha)H} \mathrm{d} r  =\frac{C}{n} \|f\|_{\cC^\alpha}  (t-s)^{\alpha H}.\label{est:I2-3}
       \end{equation}
       Then~\eqref{eq:A_st:t-ssmall} togehter with \eqref{est:I1-3} and~\eqref{est:I2-3} verifies \eqref{sew:A} of \cref{lem:sewing-2} with taking $C_1=\frac{C}{n} \|f\|_{\cC^\alpha} $, since $((1-(1-\alpha)H)\wedge(\alpha H\big))>\gamma>\frac{1}{2}$. 

       Next we verify \eqref{sew:deltaA}. Let $(s,u,t)\in\overline{[0,1]}_1^3$. Recall the definition of $s_i,i=1,\dots,6$ in~\eqref{eq:points}. Similarly to~\eqref{def:deltaA}, we have 
       \begin{align*}
           \E^{s_1} \delta A_{s,u,t} 
           &=\E^{s_1} \int_{s_4}^{s_5} \mathcal{P}^H_{r-s_3} f(\E^{s_3} B_r^H + \E^{s_1} \varphi_{k_n(r)}^n)-\mathcal{P}^H_{r-s_3} f(\E^{s_3} B_{k_n(r)}^H + \E^{s_1} \varphi^n_{k_n(r)}) \\
           & \qquad -\mathcal{P}_{r-s_3}^H f(\E^{s_3} B^H_r+ \E^{s_3} \varphi_{k_n(r)}^n) + \mathcal{P}^H_{r-s_3} f (\E^{s_3} B^H_{k_n(r)} + \E^{s_3} \varphi^n_{k_n(r)}) \mathrm{d}r \\
           &\quad +\E^{s_1} \int_{s_5}^{s_6} \mathcal{P}^H_{r-s_2} f(\E^{s_2} B^H_r + \E^{s_1} \varphi^n_{k_n(r)}) - \mathcal{P}^H_{r-s_2} f(\E^{s_2} B^H_{k_n(r)} + \E^{s_1} \varphi^n_{k_n(r)} )\\
           &\qquad - \mathcal{P}^H_{r-s_2} f(\E^{s_2} B^H_r + \E^{s_2} \varphi_{k_n(r)}^n)+\mathcal{P}^H_{r-s_2} f(\E^{s_2} B^H_{k_n(r)} + \E^{s_2} \varphi^n_{k_n(r)} ) \mathrm{d} r \\
           & =: V_1+ V_2.
       \end{align*}
       Again we observe that the above two terms can be treated in the  exactly same way, so we only detail $V_1$.
       
       Applying~\eqref{eq:heat2} with taking 
       \begin{align*}
           &x_1=\E^{s_3} B_{k_n(r)}^H + \E^{s_1} \varphi^n_{k_n(r)}, x_2=\E^{s_3} B_r^H + \E^{s_1} \varphi_{k_n(r)}^n,
          \\& x_3=\E^{s_3} B^H_{k_n(r)} + \E^{s_3} \varphi^n_{k_n(r)},
           x_4= \E^{s_3} B^H_r+ \E^{s_3} \varphi_{k_n(r)}^n,
       \end{align*}
       we get
       \begin{align}
         \label{Cor4.4-V1}
        &   \| V_1 \|_{L^p_\omega} \nonumber
           \\&\le C \| f \|_{\cC^\alpha}\int_{s_4}^{s_5}  (r-s_3)^{-H(2-\alpha)} \| \E^{s_1} [| \E^{s_3}(B_r^H-B_{k_n(r)}^H)| \cdot |\E^{s_1} \varphi^n_{k_n(r)}-\E^{s_3} \varphi^n_{k_n(r)}|] \|_{L^p_\omega}\mathrm{d} r .
          \end{align}
    By Cauchy-Schwarz inequality, 
       \begin{align*}
           &\E^{s_1} [| \E^{s_3}(B_r^H-B^H_{k_n(r)}) | \cdot |\E^{s_1} \varphi^n_{k_n(r)}-\E^{s_3} \varphi^n_{k_n(r)}|] \\
           & \le (\E^{s_1} (\E^{s_3} (B_r^H-B^H_{k_n(r)}))^2)^{\frac{1}{2}} (\E^{s_1} ( \E^{s_1} \varphi_{k_n(r)}^n -\E^{s_3} \varphi^n_{k_n(r)})^2 )^{\frac{1}{2}}.
       \end{align*}
    By Jensen inequality, we obtain
       \begin{align*}
           (\E^{s_3} (B_r^H-B^H_{k_n(r)}))^2 \le &\E^{s_3} (B_r^H-B_{k_n(r)}^H)^2,\\
           ( \E^{s_1} \varphi^n_{k_n(r)} -\E^{s_3} \varphi^n_{k_n(r)} )^2 
           =& (\E^{s_3} ( \varphi^n_{k_n(r)} -\E^{s_1} \varphi^n_{k_n(r)}))^2  \le \E^{s_3} ( \varphi^n_{k_n(r)} - \E^{s_1} \varphi^n_{k_n(r)} )^2.
       \end{align*}
    Therefore, by ~\eqref{eq:phi-1}, we have
       \begin{align*}
           &\E^{s_1} [| \E^{s_3}(B_r^H-B^H_{k_n(r)}) | \cdot |\E^{s_1} \varphi^n_{k_n(r)}-\E^{s_3} \varphi^n_{k_n(r)}|] \\
           & \le ( \E^{s_1} (B^H_r-B^H_{k_n(r)})^2)^{\frac{1}{2}} \cdot (\E^{s_1} (\varphi^n_{k_n(r)} -\E^{s_1} \varphi^n_{k_n(r)})^2)^{\frac{1}{2}} \\
           & \le C (\E^{s_1} (B^H_r-B^H_{k_n(r)})^2)^{\frac{1}{2}} \cdot (k_n(r)-s_1)^{1+\alpha H}.
       \end{align*}
    For $p \ge 2$, ~\eqref{eq:CJI} and~\eqref{eq:rep} imply
       \begin{equation} \label{Cor4.4-p>2}
           \begin{aligned}
                & \| \E^{s_1} [| \E^{s_3} (B_r^H-B^H_{k_n(r)})| | \E^{s_1} \varphi^n_{k_n(r)}-\E^{s_3} \varphi^n_{k_n(r)}|] \|_{L^p_\omega} \\
                & \le C (r-s_1)^{1+\alpha H} \left\| \| B_r^H-B^H_{k_n(r)} \|_{L_\omega^2 | \mathcal{F}_{s_1}} \right\|_{L^p_\omega} \\
                & = C (r-s_1)^{1+\alpha H} \| \E^{s_1} | B_r^H- B_{k_n(r)}^H|^2 \|^{\frac{1}{2}}_{L_\omega^{\frac{p}{2}}}\\
                & \le C (r-s_1)^{1+\alpha H} \| B_r^H -B^H_{k_n(r)} \|_{L^p_\omega} \\
                & \le \frac{C}{n} (r-s_1)^{1+\alpha H} .
           \end{aligned}
       \end{equation}
    Therefore, plugging ~\eqref{Cor4.4-p>2} into ~\eqref{Cor4.4-V1}, we have
       \begin{align*}
           \|V_1 \|_{L^p_\omega} 
           & \le  \frac{C}{n} \|b\|_{\cC^\alpha} \int_{s_4}^{s_5} (r-s_3)^{-(2-\alpha)H} (r-s_1)^{1+\alpha H} \mathrm{d} r \le  \frac{C}{n} \|f\|_{\cC^\alpha} |t-s|^{2+(2\alpha-2)H}.
       \end{align*}
      The same bound holds on $V_2$.

   Since ~\cref{ass:main1} implies $2+(2 \alpha -2)H>1$,
   we can obtain that ~\eqref{sew:deltaA} holds with $C_2= \frac{C}{n} \|b\|_{\cC^\alpha} $.

  Similarly to~\cref{lem:E1}, we could verify the process $\mathcal{A}$ in~\eqref{eq:sew:A_bound} is given by 
         \begin{equation*}
             \mathcal{A}_t=\int_0^t f(B^H_r+ \varphi^n_{k_n(r)})-f(B^H_{k_n(r)}+\varphi_{k_n(r)}^n) \mathrm{d} r.
        \end{equation*}

    In the end all of the conditions from  \cref{lem:sewing-2} are verified, which  proves the desired result.
\end{proof}

With \cref{lem:E1}, \cref{lem:E2} and \cref{lem:E3} at hand we are ready to give:
  \begin{proof}[Proof of \cref{thm:main}]
   
    By \eqref{est:E1}, \eqref{est:E2} and \eqref{est:E3}, we see that
       \begin{align*}
         \| ( X-X^n)_t-(X-X^n)_s \|_{L^p_\omega}    &\le C ( \| \varphi-\varphi^n \|_{C^{\gamma}_p[s,t]} + n^{-1}) |t-s|^{\gamma+\varepsilon}\\& = C ( \| X-X^n \|_{C^{\gamma}_p[s,t]} + n^{-1}) |t-s|^{\gamma+\varepsilon},
       \end{align*}
    which implies
       \begin{align*}
           [X-X^n]_{C_p^{\gamma}[S,T]} \le C (\| X-X^n \|_{C_p^{\gamma}[S,T]} + n^{-1} ) (T-S)^\varepsilon.
       \end{align*}
    Therefore, we have
      \begin{align*}
          \| X-X^n \|_{C_p^{\gamma}[S,T]} & \le | (X-X^n)_S| + 2[X-X^n]_{C_p^{\gamma}[S,T]} \\
          & \le  | (X-X^n)_S| +C  (\| X-X^n \|_{C_p^{\gamma}[S,T]} + n^{-1} ) (T-S)^\varepsilon.
      \end{align*}
    Fix $T-S=\Delta$ small enough and we obtain
        \begin{equation*}
            \|X-X^n \|_{C^{\gamma}_p[S,T]} \le C (|(X-X^n)_S| + n^{-1}).
      \end{equation*}
    Dividing $[0,1]$ into $[0, \Delta], [\Delta, 2 \Delta], \ldots$, yields that
      \begin{equation*}
          \|X-X^n \|_{C^{\gamma}_p[0,1]} \le C(|x_0-x_0^n|+ n^{-1}).
      \end{equation*}
\end{proof}
 \section{Optimality} \label{sec:optimal}
 In this section, we discuss the optimality of the convergence rate $n^{-1}$ obtained in the previous part. By showing that $n(X-X^n)$ converges to a non-zero limit, we verify that the rate $n^{-1}$ is optimal for the scheme \eqref{eq:SDE-EM}.

 In the following first we present a regularization lemma concerning the solution to \eqref{eq:SDE}.

  \begin{lemma} \label{lem:reg}
Let $(X_t)_{t\in[0,1]}$ be the solution to \eqref{eq:SDE}.  Suppose \cref{ass:main1} holds. Then for every $p>1$, for any  $f\in \cC^1$, we have
        \begin{equation}\label{est:reg}
            \Big\| \int_0^{\cdot} \nabla f(X_t) \mathrm{d} t  \Big\|_{C_p^{1+(\alpha-1)H}[0,1]} \leq C\|f\|_{\cC^\alpha},
        \end{equation}
   where $C$ depends only on $d,p,H,\alpha, \|b\|_{\cC^\alpha}$.     
 \end{lemma}
 \begin{proof}
     In order to show \eqref{est:reg}, we apply \cref{lem:sewing-2}. Let  $M=1$, $(s,t)\in[0,1]_1^2$ and 
        \begin{equation*}
            A_{s,t}:=\E^{s-(t-s)} \int_s^t \nabla f(B_r^H+ \E^{s-(t-s)} \varphi_r) \mathrm{d} r \text{ where } \varphi:=X-B^H.
        \end{equation*}
     We are going to verify \eqref{sew:A} and \eqref{sew:deltaA}. By~\eqref{eq:EBt}, we see
        \begin{equation*}
            A_{s,t} = \int_s^t \mathcal{P}^H_{r-[s-(t-s)]} (\nabla f)(\E^{s-(t-s)} B_r^H + \E^{s-(t-s)} \varphi_r) \mathrm{d} r.
        \end{equation*}
     Then by ~\eqref{est:heat-1}, we have 
\begin{align*}
        |A_{s,t}|  \le \int_s^t \| \mathcal{P}^H_{r-[s-(t-s)]} (\nabla f) \|_{C^0} \mathrm{d}r 
                 \le C \int_s^t  \left[ r-[s-(t-s)] \right]^{H(\alpha-1)} \| f \|_{C^{\alpha}} \mathrm{d} r.
\end{align*}
     Therefore, we have $ \|A_{s,t}\|_{L^p_\omega} \le C \|f\|_{C^\alpha} (t-s)^{1-(1-\alpha)H }.$
     Then~\eqref{sew:A} holds  with $C_1=C \|f\|_{C^\alpha}$ by the fact that $1-(1-\alpha)H > \frac{1}{2}$.
       
     Next we verify~\eqref{sew:deltaA}. Let $(s,u,t)\in\overline{[0,1]}_1^3$. Recall the definition of $s_i,i=1,\dots,6$ in~\eqref{eq:points}. We first can write
       \begin{equation*}
           \begin{aligned}
               \E^{s_1} \delta A_{s,u,t} & 
               =\E^{s_1} \bigg[ \int_{s_4}^{s_5} \E^{s_1} \nabla f(B_r^H+ \E^{s_1} \varphi_r)-\E^{s_3} \nabla f(B_r^H+ \E^{s_3} \varphi_r) \mathrm{d}r \bigg] \\
               & \quad + \E^{s_1} \bigg[ \int_{s_5}^{s_6} \E^{s_1} \nabla f(B_r^H+ \E^{s_1} \varphi_r) -\E^{s_2} \nabla f(B_r^H+ \E^{s_2} \varphi_r ) \mathrm{d} r\bigg] \\
               & = \int_{s_4}^{s_5} \E^{s_1} \bigg[ \mathcal{P}^H_{r-s_3} \nabla f(\E^{s_3} B_r^H+ \E^{s_1} \varphi_r) - \mathcal{P}^H_{r-s_3} \nabla f (\E^{s_3} B_r^H + \E^{s_3} \varphi_r)\bigg] \mathrm{d} r \\
               & \quad + \int_{s_5}^{s_6} \E^{s_1} \bigg[ \mathcal{P}^H_{r-s_2} \nabla f (\E^{s_2} B_r^H+ \E^{s_1} \varphi_r) - \mathcal{P}^H_{r-s_2} \nabla f (\E^{s_2} B_r^H+ \E^{s_2} \varphi_r) \bigg] \mathrm{d}r.
           \end{aligned}
       \end{equation*}
 By~\eqref{est:heat-1} and~\eqref{eq:phi}, we have
       \begin{equation*}
           \begin{aligned}
               \|\E^{s_1} \delta A_{s,u,t}\|_{L^p_\omega} 
               & \le C \int_s^t (r-s_3)^{H(\alpha-2)} \| f \|_{\cC^{\alpha}}\|\E^{s_1} [|\E^{s_1} \varphi_r-\E^{s_3} \varphi_r|]\|_{L^p_\omega} \mathrm{d}r \\
               & \quad + \int_u^t (r-s_2)^{H(\alpha-2)} \| f \|_{\cC^{\alpha}} \|\E^{s_1} [|\E^{s_1} \varphi_r-\E^{s_2} \varphi_r|]\|_{L^p_\omega} \mathrm{d}r\\
               &\le C \|f\|_{\cC^\alpha} |t-s|^{2+2\alpha H-2H}.
           \end{aligned}
       \end{equation*}
      Noticing  ~\cref{ass:main1} implies $2+2\alpha H-2H>1$, we conclude ~\eqref{sew:deltaA} holds with $C_2= C \|f\|_{\cC^\alpha} $.
      Moreover, since $f\in C^1$, it is obviously to see that the process $\mathcal{A}$ in~\eqref{eq:sew:A_bound} actually is given by 
        \begin{equation*}
            {\mathcal{A}}_t= \int_0^t \nabla f(X_r) \mathrm{d} r.
        \end{equation*}
    Therefore, by ~\cref{lem:sewing-2}, we have 
        \begin{equation*}
            \big\| \int_s^t \nabla f(X_r) \mathrm{d} r \big\|_{L^p} \le C\|f\|_{\cC^\alpha}|t-s|^{H(\alpha-1)+1},
        \end{equation*}
    which implies that \eqref{est:reg} holds.
 \end{proof}
 Due to \eqref{est:reg}, we can already state a result which is analogous to \cite[Proposition 2.8]{DGL-CLT}. 
 \begin{proposition} \label{prop:reg}
    Let $(X_t)_{t\in[0,1]}$ be the solutions to \eqref{eq:SDE}. Suppose \cref{ass:main1} holds. For every $f\in \cC^\alpha(\mR^d)$ there exists a $(\cF_t)_{t\in[0,1]}$-adapted process $(\cH f)^X_{t\in[0,1]}\in C_p^{1+(\alpha-1)H-}[0,1]$ for every $p\ge 2$ such that for $g\in \cC^1$ with probability one, 
         \begin{align*}
             (\cH g)^X_t=\int_0^t\nabla g(X_s)\mathrm{d} s,\quad t\in[0,1].
         \end{align*}
 \end{proposition}
\begin{proof}
    In fact, we notice that for any $f \in \cC^{\alpha},$ there exists sequence $(f_n)_{n \in \mathbb{N}} \subset \cC^1$, such that $\lim_{n \to \infty} f_n=f$ in $\cC^{\alpha-}.$ 
    Therefore, by~\cref{lem:reg}, we have $\big( \int_0^{\cdot} \nabla f_n(X_t) \mathrm{d}t \big)_{n \in \mathbb{N}}$ is a Cauchy sequence in $C_p^{1+(\alpha-1)H-}[0,1].$ We define $\cH f=\lim_{n\rightarrow\infty}\int_0^{\cdot} \nabla f_n(X_t) \mathrm{d}t$ in $C_p^{1+(\alpha-1)H-}[0,1].$
\end{proof}
   Therefore, for $b \in \cC^{\alpha}$, we can define $\int_0^{\cdot} \nabla b(X_t) \mathrm{d} t:=\lim_{n \to \infty} \int_0^{\cdot} \nabla b_n(X_t) \mathrm{d} t$ in probability. Moreover from Kolmogorov continuity criteria we know that it has a version which has $\theta$-H\"older continuous path for $\theta\in(0, 1+(\alpha-1)H)\subset(0,1)$.

Now we are at the position of showing \cref{thm:main1}.
 \begin{proof}[Proof of \cref{thm:main1}]
     Clearly \cref{prop:reg} shows the existence of $(\cH f)^X\in C_p^{1+(\alpha-1)H-}[0,1]$ which is bounded linear for $f\in \cC^{\alpha}$, it also implies that the $c$ satisfying \eqref{opt-ODE} for given $X$ is well-defined and $\|c\|_{\cC^{1+(\alpha-1)H-}([0,1])}<\infty,\mP$-a.s. Indeed  from  Kolmogorov continuity criteria and the property of Young integral we know that for any possible solution $c$, for any $k_1\in(\frac{1}{2},{1+(\alpha-1)H}),k_2\in(\frac{1}{2},(1+(\alpha-1)H)\wedge\alpha)=(\frac{1}{2},{1+(\alpha-1)H})$ so that $k_1+k_2>1$, for any $0\leq s\leq t\leq 1$ and $q>p$, we have
     \begin{align}\label{reg:c}
         \|c(t)-c(s)\|_{L^p_\omega}\le &\| \int_s^tc(r)\mathrm{d}(\cH b)_r^{X}\|_{L^p_\omega}+\frac{1}{2}\|b(X_t)-b(X_s)\|_{L^p_\omega} \nonumber\\\le C&\big\|\|c\|_{\cC^{k_1}([0,1])}\|(\cH b)^{ X}\|_{\cC^{k_2}([0,1])}\big\|_{L^p_\omega}|t-s|^{k_2}+C\|b\|_{\cC^\alpha(\mR^d)}|t-s|^{\alpha}
 \nonumber\\\le C&\|c\|_{L_\omega^q(\cC^{k_1}([0,1]))}\|(\cH b)^{ X}\|_{L_\omega^{\frac{qp}{q-p}}(\cC^{k_2}([0,1]))}|t-s|^{k_2}+C\|b\|_{\cC^\alpha(\mR^d)}|t-s|^{k_2}
  \nonumber \\\le C&\|b\|_{\cC^\alpha(\mR^d)}|t-s|^{k_2}.
     \end{align}
Then, by   Kolmogorov continuity criteria again we get $\mP$-a.s. $\|c\|_{\cC^{1+(\alpha-1)H-}}<\infty$.  Notice that \eqref{opt-ODE} is a linear ODE in the sense of Young integral for given $X$, standard fixed point arguments gives us the existence and uniqueness of $c\in \cC^{1+(\alpha-1)H-}([0,1]) $ satisfying \eqref{opt-ODE}.       
     
     Therefore we only need to show that \eqref{eq:optimal} holds for such $c$  and   for $b\in\cC^\alpha$, $\alpha<1$. 

     Denote $c^n:=n(X-X^n)$. 
 \cref{thm:main} implies that $  \sup_n \|c_n\|_{\cC_p^{1+(\alpha-1)H-}[0,1]}<\infty,$ that is to say, for $\gamma<1+(\alpha-1)H$, from Kolmogorov continuity criteria, $\sup_n \|c_n\|_{L^p(\Omega,\cC^{\gamma})}<\infty$, hence for $\cL^n:=(X,X^n,c^n, c, c^n-c, B^H)$, $ \sup_n \|\cL^n\|_{L^p(\Omega,(\cC^{\gamma})^6)}<\infty.$

By Sobolve embedding we know that there exists large $q$ and $\gamma_1<\gamma_2<\gamma$ so that compactly \begin{align}
    \label{emd:com}\cC^{\gamma}([0,1],\mR^d)\subset\subset W_q^{\gamma_2}([0,1],\mR^d)\subset \subset\cC^{\gamma_1}([0,1],\mR^d),
\end{align}
hence the laws of $(\cL^n)_{n\in\mN}$ are tight on $W_q^{\gamma_2}([0,1],\mR^d)^6$ which is Polish 
 and 
\begin{align}
    \label{uniform}
    \sup_n\||\cL^n|\|_{L^p(\Omega, W_q^{\gamma_2}([0,1],\mR^d)^6)}<\infty.
\end{align}Following from  Prokorov’s theorem, we have that the laws of a further subsequence
 $(\cL^n)_{n\in\mN_1}, \mN_1\subset\mN$, converges weakly on $W_q^{\gamma_2}([0,1],\mR^d)^6$. By Skorohod’s representation theorem, there exists a probability space $(\hat\Omega,\hat{\cF}, \hat\mP)$ and random variables 
 $\hat \cL^n:=(\tilde X^n,\hat X^{n,n}, \hat c^n, \tilde c^n,\hat c^n-\tilde c^n,\hat B^{H,n}):\hat\Omega\mapsto W_q^{\gamma_2}([0,1],\mR^d)^6$ and  $\cL:=(\tilde X,\hat X, \hat c, \tilde c, \hat c-\tilde c,  \hat B^H):\hat\Omega\mapsto W_q^{\gamma_2}([0,1],\mR^d)^6$  so that 
 \begin{align}\label{eq:samelaw}
     \cL^n \text{ on } (\Omega,\cF,\mP)\overset{d}{=}\hat\cL^n  \text{ on }(\hat\Omega,\hat{\cF}, \hat\mP),
 \end{align}
and 
\begin{align}
    \label{con:p-a.s.}\hat\mP-a.s.,\quad (\hat\cL^n)_{n\in\mN_1} \quad \text{converges to } \quad \cL \quad\text{ in }   W_q^{\gamma_2}([0,1],\mR^d)^6.
\end{align}
Let $(\hat\cF_t)_{t\geq0}$ be the augmentation of the filtration generated by $(\hat X, \hat X^n, \hat B^{H,n})$. \eqref{eq:samelaw}, the continuity of $b$ and the strong well-posedness of \eqref{eq:SDE} imply 
\begin{align*}
    \tilde X_t&=x_0+\int_0^tb(\tilde X_s)\mathrm{d}s+\hat B_t^H, \quad  \tilde X_t^k=x_0+\int_0^tb(\tilde X_s^k)\mathrm{d}s+\hat B_t^{H,k},\\
    \hat X_t^{k,n}&=x_0+\int_0^tb(\hat X_{k_n(s)}^{k,n})\mathrm{d}s+\hat B_t^{H,k},\quad
    \hat c^n=n(\tilde X^n- \hat X^{n,n}),
    \\
    \tilde c^n(t)&= \int_0^t\tilde c^n(s)\mathrm{d} (\cH b)^{\tilde X_s^n}+ \frac 12 (b(\tilde  X_t^n)-b(x_0)), \quad t\in[0,1].
\end{align*}
Following from \eqref{eq:samelaw} which implies the uniform integrability, \eqref{con:p-a.s.} and Vitali's theorem, for all $p\geq1$ we have 
\begin{align}\label{con:L-Ln}
    \lim_{n\rightarrow\infty}\big\|\hat\cL^n-\cL\big\|_{L^p(\hat\Omega, W_q^{\gamma_2}([0,1],\mR^d)^6)}=0.
\end{align}
At this point we can also claim that $\tilde{X}=\hat X$ $\hat\mP$-a.s. from \cref{thm:main} and the strong well-posedness of \eqref{eq:SDE}. 

 Notice that for any $b \in \cC^{\alpha}$,  there exists a sequence $\{b_n\}_{n \in \mathbb{N}} \subset \cC^1$ such that $\lim_{n \to \infty} b_n=b$ in $\cC^{\alpha-}$, 
 then for $m\in\mN$
 \begin{align}
    \hat c^n(t)=&
    n(\tilde X_t^n-\hat X_t^{n,n})=
    n \left( \int_0^t b( \tilde X_s^n) \mathrm{d}s-\int_0^t b( \hat X_{k_n(s)}^{n,n}) \mathrm{d}s \right)\nonumber\\=&
    n \int_0^t (b-b_m)( \tilde X_s^n) - (b-b_m)( \hat X_{k_n(s)}^{n,n})\mathrm{d}s +n\int_0^tb_m( \tilde X_r^n)-b_m( \hat X_r^{n,n})\mathrm{d}s\nonumber\\&+n\int_0^tb_m( \hat X_r^{n,n})-b_m( \hat X_{k_n(s)}^{n,n})\mathrm{d}s=:
    J_t^{m,n}+K_t^{m,n}+L_t^{m,n}.\label{def:Vn}
 \end{align}
Following from \cref{lem:E1}, \cref{lem:E2} and \cref{lem:E3} we get 
\begin{align}\label{est:J}
    \big\|\sup_{t\in[0,1]}|J_t^{m,n}|\big\|_{L_\omega^p}\lesssim\|b-b_m\|_{\cC^{\alpha-}};
\end{align}
   for $K^{m,n}$  
   \begin{align*}
 K_t^{m,n}=\int_0^t\int_0^1\nabla b_m(\theta  \tilde X_r^{n}+(1-\theta) \hat X_r^{n,n})\cdot  \hat c_r^n\mathrm{d}\theta\mathrm{d}r
,
   \end{align*} 
   since $ \hat\mP$-a.s. $\hat c^n$ converges to $\hat c$, $\hat X^{n,n}$ converges to $\hat X$, and $\tilde X^n$ converges to $\tilde X=\hat X$,  $ \hat\mP$-a.s. we have    
   \begin{align}\label{est:K}
       \lim_{n\rightarrow\infty} K_t^{m,n}=
       \int_0^t\nabla b_m(\hat X_r)\cdot  \hat c_r\mathrm{d}r= \int_0^t  \hat c_r\mathrm{d}(\cH b_m)_r^{\hat X},\quad \hat\mP-a.s.
   \end{align}
 Concerning $L^{m,n}$, following from  \eqref{eq:SDE-EM} and the fact a.s. $B^H\in \cC^{H-\eps}([0,1])$ for sufficiently small $\eps>0$ by Kolmogorov continuity criteria, which means there exists a small enough $\varepsilon'>0$ so that a.s. $B_H\in\cC^{1+\varepsilon'}$,  we have a.s. (denote {$\{c\}:=c-\lfloor c\rfloor$ for $c\in\mR_+$})
         \begin{equation*}
            \begin{aligned}
                n(\hat X_{r}^n-\hat X_{k_ n(r)}^{n,n})&=n(r-k_ n(r))b(\hat X_{k_ n(r)}^{n,n})+n(\hat B_r^{H,n}-\hat B_{k_ n(r)}^{H,n})\\
                &=\{nr\}b(\hat X_{k_ n(r)}^{n,n})+\{nr\}(\hat B_r^{H,n})'+o(n^{-\eps'});
            \end{aligned}
        \end{equation*}  
        together with the condition that  $\nabla b_m$ is continuous and bounded, moreover $\hat X_{r}^{n,n}\rightarrow \hat X_{r}$ and $\hat X_{k_n(r)}^{n,n}\rightarrow \hat X_{r}$ a.s. from \cref{thm:main},
dominated convergence theorem shows as $n\rightarrow\infty$, we have  
   $ \hat\mP$-a.s. 
   \begin{align}\label{est:L}
    &\lim_{n\rightarrow \infty}L_t^{m,n} \nonumber\\&=    \lim_{n\rightarrow \infty}n\int_0^t\int_0^1\nabla b_m(\theta  \hat X_r^{n,n}+(1-\theta) \hat X_{k_n(r)}^{n,n})\cdot ( \hat X_r^{n,n}- \hat X_{k_n(r)}^{n,n})\mathrm{d}\theta\mathrm{d}r
    \nonumber\\&=\lim_{n\rightarrow \infty}\int_0^t\int_0^1\nabla b_m(\theta  \hat X_r^{n,n}+(1-\theta) \hat X_{k_n(r)}^{n,n})\Big(\{nr\}\big(b(\hat X_{k_ n(r)}^{n,n})+(B_r^H)'\big)+o(n^{-\eps'})\Big)\mathrm{d}r \nonumber\\&=\frac{1}{2}\int_0^t\nabla b_m(\hat X_r)\big(b( \hat X_r)+(\hat B_r^H)'\big)\mathrm{d}r=
    \frac{1}{2}\big(b_m(\hat X_t)-b_m(x_0)\big),
   \end{align}
  we have applied the fact that  $\{n\cdot\}$ converges to $\frac 12$ weakly in $L^2([0,t])$ as $n\rightarrow\infty$ in deriving the  penultimate equality and the lase equality holds due to elementary chain rule.

Now we put \eqref{est:J},  \eqref{est:K} and \eqref{est:L} together with taking $m,n\rightarrow\infty$ for both side of \eqref{def:Vn}  then get $ \hat\mP$-a.s. 
\begin{align*}
\hat c(t) 
=&\int_0^t  \hat c_r\mathrm{d}(\cH b)_r^{\hat X} +\frac{1}{2}(b(\hat X_t)-b(x_0)).
\end{align*}
That is, $(\hat X, \hat c)$ satisfy \eqref{eq:SDE} and \eqref{opt-ODE}.

Meanwhile following from the fact that $\hat \mP$ a.s.,  $\tilde X^n$ converges to $\hat X$, $\tilde c^n$ converges to $\tilde c$  and the property of Young's integral \eqref{Young-est} 
we can conclude $\tilde c$ also satisfies \eqref{opt-ODE} for such $\hat X$. It implies that $\hat c =\tilde c$ a.s. and  $\hat\mP$-a.s., $\hat c^n-\tilde c^n$ converges to $0$; hence we get that for given $X$, $c^n$ converges to $c$ in probability on $(\Omega,\cF,\mP)$. 

\end{proof}
\appendix
\section{Auxiliary Proofs} \label{app}
\begin{proof}
    [Proof of \eqref{est:E2-app-1} and  \eqref{est:E2-app-2}]\label{app:E2}
      Similarly to~\eqref{eq:A_st_heat} and~\eqref{eq:AstLp}, we have
        \begin{align}
           \| A_{s,t} \|_{L^p_\omega}&= \| \int_s^t \mathcal{P}^H_{r-[s-(t-s)]} f (\E^{s-(t-s)} B^H_r + \E^{s-(t-s)} \varphi_{r}^n) \nonumber\\
        & \qquad -\mathcal{P}^H_{r-[s-(t-s)]} f(\E^{s-(t-s)} B^H_r+\E^{s-(t-s)}\varphi^n_{k_n(r)}) \mathrm{d} r\|_{L^p_\omega}
          \nonumber \\& \le C \|f\|_{\cC^\alpha} \sup_{r \in [s,t]} \| \varphi_r^n-\varphi_{k_n(r)}^n \|_{L_\omega^p} |t-s|^{1-(1-\alpha)H} \nonumber\\
           & \le C n^{-1} \|f\|_{\cC^\alpha} |t-s|^{1-(1-\alpha)H} ,\label{est:A-2}
       \end{align}
  where in the second inequality we used $\| \varphi_r^n-\varphi_{k_n(r)}^n \|_{L_\omega^p}\le \|b\|_{\cC^0} n^{-1}$. Hence \eqref{est:E2-app-1} holds. 
  
    Next we verify \eqref{est:E2-app-2}. Let $(s,u,t)\in\overline{[0,1]}_1^3$. Recall the definition of $s_i,i=1,\dots,6$ in~\eqref{eq:points}. Similarly to~\eqref{def:deltaA}, we can write 
       \begin{align*}
           \E^{s_1} \delta A_{s,u,t} =&\E^{s_1} \int_{s_4}^{s_5} \mathcal{P}_{r-s_3}^H f(\E^{s_3} B_r^H + \E^{s_1} \varphi_r^n)- \mathcal{P}_{r-s_3}^H f(\E^{s_3} B_r^H + \E^{s_1} \varphi_{k_n(r)}^n)  \\
           & \qquad\qquad- \mathcal{P}^H_{r-s_3} f(\E^{s_3} B_r^H + \E^{s_3} \varphi_r^n)+\mathcal{P}_{r-s_3}^H f(\E^{s_3} B_r^H+ \E^{s_3} \varphi_{k_n(r)}^n)\mathrm{d} r \\
           & +\E^{s_1} \int_{s_5}^{s_6} \mathcal{P}_{r-s_2}^H f(\E^{s_2} B_r^H + \E^{s_1} \varphi_r^n)- \mathcal{P}_{r-s_2}^H f(\E^{s_2} B_r^H+ \E^{s_1} \varphi_{k_n(r)}^n)  \\
           & \qquad\qquad-\mathcal{P}_{r-s_2}^H f(\E^{s_2} B_r^H + \E^{s_2} \varphi_r^n)+ \mathcal{P}_{r-s_3}^H f(\E^{s_2} B_r^H+\E^{s_2} \varphi_{k_n(r)}^n)\mathrm{d} r \\
            =&:I_1+I_2.
       \end{align*}
     The two terms are treated in  the exactly same manner, so we only detail $I_1$. Similarly to~\eqref{eq:I_1_bound}, 
    we get
       \begin{align}
           \| I_1 \|_{L_\omega^p} \le &C \| f\|_{\cC^\alpha}\int_{s_4}^{s_5} (r-s_3)^{-(1-\alpha)H}\big\| \E^{s_1} | \E^{s_3} \varphi_r^n -\E^{s_3} \varphi_{k_n(r)}^n -\E^{s_1} \varphi_r^n + \E^{s_1} \varphi_{k_n(r)}^n | \big\|_{L_\omega^p} \nonumber\\
           & \qquad+ (r-s_3)^{-(2-\alpha)H} \big\| | \E^{s_1} \varphi_r^n -\E^{s_1} \varphi_{k_n(r)}^n | \cdot \E^{s_1} | \E^{s_3} \varphi_{k_n(r)}^n-\E^{s_1} \varphi_{k_n(r)}^n | \big\|_{L_\omega^p} \mathrm{d} r.\label{eq:I_1_bound_2}
       \end{align}
    Similar to~\eqref{eq:E1E3=E1}, we have
      \begin{equation}\label{eq:E1E3=E1_2}
            \| \E^{s_1}|\E^{s_3} \varphi_r^n-\E^{s_3} \varphi_{k_n(r)}^n-\E^{s_1} \varphi_r^n+ \E^{s_1} \varphi_{k_n(r)}^n | \|_{L^p_\omega} \le \| \E^{s_1} | (\varphi_r^n-\varphi_{r}^n)-\E^{s_1} (\varphi_r^n-\varphi_{r)}^n) | \|_{L^p_\omega}.
      \end{equation}
    We note that
      \begin{align*}
          \varphi_r^n-\varphi_{k_n(r)}^n &=\int_{k_n(r)}^r b(B^H_{k_n(t)}+\varphi_{k_n(t)}^n) \mathrm{d}t =(r-k_n(r)) b(B_{k_n(r)}^H+\varphi^n_{k_n(r)}) \in \mathcal{F}_{k_n(r)}.
      \end{align*}
    When $s_1 \in [k_n(r),r],$ we have $$\varphi_r^n-\varphi_{k_n(r)}^n-\E^{s_1}(\varphi_r^n-\varphi_{k_n(r)}^n)=0;$$
    when $s_1 < k_n(r),$ by taking $X=b(B_{k_n(r)}^H+\varphi^n_{k_n(r)})$ and $Y=b(\E^{s_1} B_{k_n(r)}^H+\varphi^n_{s_1}) \in \mathcal{F}_{s_1}$ in ~\eqref{eq:condition}, we obtain
      \begin{align*}
         &\E^{s_1} | (\varphi_r^n-\varphi_{k_n(r)}^n)-\E^{s_1} (\varphi_r^n-\varphi_{k_n(r)}^n)| \\
          =&(r-k_n(r)) \E^{s_1} |b(B_{k_n(r)}^H + \varphi_{k_n(r)}^n)-\E^{s_1} b(B_{k_n(r)}^H + \varphi_{k_n(r)}^n)| \\
         \le& 2 (r-k_n(r)) \E^{s_1} | b(B^H_{k_n(r)}+ \varphi_{k_n(r)}^n)-b(\E^{s_1} B_{k_n(r)}^H + \varphi_{s_1}^n)|\\
         \le& C(r-k_n(r))  \E^{s_1} (|B_{k_n(r)}^H-\E^{s_1} B_{k_n(r)}^H|^\alpha + | \varphi_{k_n(r)}^n-\varphi_{s_1}^n|^\alpha).
      \end{align*}
    Moreover, using ~\eqref{eq:rep1} and ~\eqref{eq:phi-1}, we have
      \begin{equation*}
          \begin{aligned}
              &\E^{s_1} | (\varphi_r^n-\varphi_{k_n(r)}^n)-\E^{s_1} (\varphi_r^n-\varphi_{k_n(r)}^n)| \\
              \le& C(r-k_n(r))  (|k_n(r)-s_1|^{\alpha H} + |k_n(r)-s_1|^{\alpha}) 
              \le  \frac{C}{n} \|b\|_{\cC^\alpha} |r-s_1|^\alpha
          \end{aligned}
      \end{equation*}
    where we used the fact $H>1$ in the second inequality. Plugging it into~\eqref{eq:E1E3=E1_2}, we get
        \begin{equation} \label{eq:cor4.3-I1-1}
            \| \E^{s_1}|\E^{s_3} \varphi_r^n-\E^{s_3} \varphi_{k_n(r)}^n-\E^{s_1} \varphi_r^n+ \E^{s_1} \varphi_{k_n(r)}^n | \|_{L^p_\omega} \le  \frac{C}{n} |r-s_1|^\alpha. 
        \end{equation}
        
    Meanwhile ~\eqref{eq:phi_n_reg} and \eqref{eq:phi-1} yield 
      \begin{align} 
         \E^{s_1} |\mE^{s_3} \varphi_{r}^n-\E^{s_1} \varphi_{r}^n| \le C |r-s_1|^{1+\alpha H},\label{eq:cor4.3-I1-2}\\ 
         \| \E^{s_1} (\varphi_r^n-\varphi_{k_n(r)}^n) \|_{L_\omega^p} \le C \| \varphi_\cdot^n-\varphi_{k_n(\cdot)}^n \|_{C_p^0} \le \frac{C}{n}.\label{eq:cor4.3-I1-3}
      \end{align}
    Applying ~\eqref{eq:cor4.3-I1-1},~\eqref{eq:cor4.3-I1-2} and~\eqref{eq:cor4.3-I1-3} into~\eqref{eq:I_1_bound_2} gives us
      \begin{align} \label{est:III1}
        \| I_1 \|_{L^p_\omega} &\le C \frac{ \| f\|_{\cC^\alpha}}{n} \int_{s_4}^{s_5} (r-s_3)^{-(1-\alpha)H} (r-s_1)^\alpha \mathrm{d} r \nonumber + (r-s_3)^{-(2-\alpha)H} (r-s_1)^{1+\alpha H}\mathrm{d}r \nonumber\\
        & \le C \frac{\|f\|_{\cC^\alpha}}{n} (t-s)^{(1+\alpha-(1-\alpha)H) \wedge (2+(2 \alpha -2)H)}.
      \end{align}

     The same bound holds on $I_2$.  Hence we get \eqref{est:E2-app-2}.

\end{proof}

\section*{Acknowledgment}
CL is supported by Deutsche Forschungsgemeinschaft (DFG) - Projektnummer 563883019.
KS is grateful to the financial supports by National Key R \& D Program of China (No. 2022YFA1006300) and the financial supports of the NSFC (No. 12426205, No. 12271030). 

We are  grateful to Prof. Rongchan Zhu (Beijing Institute of Technology) for her  fruitful advices. We also are grateful for valuable suggestion on the optimality result from an anonymous reviewer.
\bibliographystyle{alpha-mod}
\bibliography{references}

    \end{document}